\documentclass[lettersize,onecolumn]{IEEEtran}
\usepackage{cite}
\usepackage{amsmath,amssymb,amsfonts}
\usepackage{algorithmic}
\usepackage{graphicx}
\usepackage{textcomp}
\usepackage{xcolor}
\usepackage{authblk}
\usepackage[linesnumbered,ruled,vlined]{algorithm2e}
\usepackage{subcaption}
\usepackage{bbm}

\newtheorem{theorem}{\textbf{Theorem}}

\newtheorem{lemma}{\textbf{Lemma}}
\newtheorem{remark}{Remark}

\newtheorem{definition}{\textbf{Definition}}

\newcommand{\tr}{\text{Tr}}
\newcommand{\cov}{\text{Cov}}
\newcommand{\rk}{\mathrm{rank}}

\newcommand{\dig}{\mathtt{diag}}

\begin{document}

\title{The Control Plant as A Communication Channel: Implicit Communication for Decentralized LQG Control}

\author{Gongpu Chen,~\IEEEmembership{Member,~IEEE} and Deniz Gündüz,~\IEEEmembership{Fellow,~IEEE}
	\thanks{The authors are with the Department of Electrical and Electronic Engineering, Imperial College London, London SW7 2AZ, UK (e-mails:\{gongpu.chen, d.gunduz\}@imperial.ac.uk).}
}



\maketitle

\begin{abstract}
We study a decentralized linear quadratic Gaussian control problem, in which a leader and a follower must steer a linear system to a target state. The target state is known only to the leader, and no explicit communication channel exists between the agents. To address the challenge posed by this asymmetric information structure, we propose an integrated communication and control (ICoCo) framework in which the control plant itself serves as a communication channel: the leader encodes the target state into its control input through an additive communication term, and the follower decodes it from the resulting state trajectory. We design an implicit coordination scheme based on joint source-channel coding ideas, and prove that the follower’s estimation error decreases monotonically to zero, enabling the two agents to coordinate increasingly well and ultimately steer the system to the target state. We then formulate the design of the communication power as an optimal control problem to minimize the overall control cost. In the fully actuated leader case, we derive necessary optimality conditions and in the under-actuated case, we solve the problem numerically. Numerical results show that the proposed scheme effectively coordinates the two agents and achieves a control cost close to that of the explicit-communication lower bound.
\end{abstract}


\section{Introduction}

\IEEEPARstart{T}{his} paper studies a decentralized linear quadratic Gaussian (LQG) control problem with asymmetric information among controllers. Specifically, we consider a setting in which a leader agent and a follower agent must coordinate to drive a linear system to a target state. Both agents observe the system state perfectly at each time step. The information asymmetry arises from the fact that the target state is known only to the leader and not to the follower. Moreover, no communication channel is available between the two agents, so the leader cannot directly transmit the target state to the follower over the horizon. Under these constraints, \textit{can the two agents still coordinate to achieve the target state with minimal cost? }

Although simple, this setting is of both practical and theoretical interest. Practically, it arises in many real-world scenarios where explicit communication is unavailable, for example due to harsh environments~\cite{harshEnv}, adversarial jamming~\cite{jamming}, or energy constraints~\cite{labidi2026joint}. Theoretically, this problem represents an interesting instance of decentralized control in which the information asymmetry lies primarily in the task-relevant information—a special information pattern for which the message to be conveyed is explicit and time-invariant. This point will be further elaborated on later.

Decentralized coordination without explicit communication channels is common in nature. A representative example is starling murmuration \cite{ballerini2008interaction}, where birds do not rely on verbal or other explicit communication channels; instead, they exchange information through their motion trajectories. By observing the movements of nearby neighbors, each bird coordinates its behavior with the flock. Similar phenomena arise in fish schooling \cite{katz2011inferring}, honeybee waggle dances \cite{seeley1999honeybee}, and firefly synchronization \cite{sarfati2021firefly}. The key feature underlying these phenomena is that animals communicate implicitly through their actions or motion trajectories in order to coordinate with one another. 

 Inspired by such natural forms of coordination, we develop an \textit{integrated communication and control (ICoCo)} framework based on the concept of implicit communication \cite{chen2025implicitcommunication}, also known as communication through actions \cite{Act2comm}. Specifically, the key idea is that \textit{a control plant can itself serve as a communication channel between a controller and a receiver that also observes the system states}: the controller (i.e., transmitter) encodes information into its control inputs by deliberately deviating from its nominally optimal policy, and the receiver decodes it from the resulting state evolution. In our setting, the leader seeks to convey the target state to the follower so that they can coordinate to drive the system to that target. Since we assume no explicit communication channel between the leader and the follower, implicit communication provides a natural solution.

The key challenge in coordinating the two agents through implicit communication is that the leader’s control input serves the dual roles of communication and control. On the one hand, the leader must transmit the target state to the follower through its control inputs; on the other hand, both agents ultimately aim to drive the system to the target state with minimal cost, which requires the leader to regulate the system appropriately. This dual role makes the design of the leader’s control input highly nontrivial.

Fortunately, thanks to the special information pattern of our problem, the message that must be transmitted to optimize the control objective is explicit,  thereby enabling the implicit communication to be designed explicitly using joint source-channel coding (JSCC)~\cite{gunduz2024joint}. In addition, \cite{chen2025implicitcommunication} establishes a separation principle for implicit communication in LQG systems: the optimal input policy balancing communication and control is given by the sum of the optimal control policy and a Gaussian signaling term. The target state is then transmitted through this signaling term. Building on this insight, we propose an implicit coordination algorithm for the two controllers. In this algorithm, the leader gradually refines the follower’s estimate of the target state via the additive Gaussian signaling term, while the follower computes its control input based on that estimate. As the estimate becomes more accurate, the follower’s control policy converges to the optimal policy for driving the system to the target state. We prove that, under the proposed algorithm, the follower’s estimation error covariance decreases monotonically to zero as time goes to infinity, provided that the covariance matrix of the signaling term is positive definite (PD) at every time step. Consequently, the follower is able to coordinate with the leader increasingly well over time and eventually steer the system to the target state. 

Our implicit coordination algorithm leaves freedom in the design of the communication power, determined by the covariance matrix of the signaling term, which is the key variable governing the tradeoff between communication and control. On the one hand, increasing the communication power raises the control cost incurred by the leader. On the other hand, greater communication power can accelerate the reduction of the follower’s estimation error, which may in turn lower the overall control cost. Since our ultimate objective is to minimize the total control cost, we formulate the communication power design problem as an optimal control problem, namely, a deterministic Markov decision process (MDP). When the leader is fully actuated (i.e., when the system is controllable by the leader alone), we derive necessary conditions for optimality under a mild assumption. When the leader is under-actuated, however, such necessary conditions are unavailable. Nevertheless, the resulting optimization problem can still be solved numerically. 

The main contributions of this paper are as follows: 
\begin{itemize}
	\item [1.] We propose an ICoCo framework in which the control plant serves as the communication channel between decentralized controllers. 
	\item [2.] We develop an implicit coordination algorithm and establish theoretical guarantees showing that the follower’s estimation error covariance decreases monotonically to zero,  provided that the signaling covariance remains positive definite, in both fully-actuated and under-actuated cases.
	\item [3.] We formulate the communication power allocation problem as an MDP, derive necessary optimality conditions in the fully-actuated case, and provide a numerical solution in the under-actuated case.
\end{itemize}
Our Simulations demonstrate that the proposed implicit communication-based scheme achieves a control cost close to that of the explicit-communication lower bound. These results highlight the potential of implicit communication as a viable mechanism for decentralized coordination in the absence of explicit communication channels.

The rest of this paper is organized as follows. Section~\ref{sec: related} discusses the related work. Section~\ref{sec:problem} presents the system model and the problem formulation based on implicit communication. Section~\ref{sec:full} develops the implicit coordination algorithm for the fully actuated leader setting, establishes the convergence guarantee for the follower's estimation error, and addresses the communication power optimization problem. Section~\ref{sec:under} extends the algorithm and theoretical result to the under-actuated leader setting. Section~\ref{sec:exp} demonstrates the experimental results. Finally, Section~\ref{sec:conclusion} concludes this paper.

\section{Related Work} \label{sec: related}
Communication and control have traditionally been treated as separate problems. In practical control systems, however, communication is not an end in itself, but a means of achieving a control objective. Consequently, communication and control are inherently coupled. 
This observation has inspired growing interest across multiple disciplines in the co-design of communication and control. In the control community, the concept of cyber-physical system has highlighted the tight integration of communication and control, motivating extensive research on their co-design \cite{TII2021,Guan2024TII,zhang2006communication}. Meanwhile, goal-oriented communication has attracted significant interest in the communication community, where communication schemes are designed to optimize the downstream task directly rather than conventional metrics such as delay or reliability \cite{Deniz2023BITS,Petar2025Goal,di2023goal}, thereby further promoting an end-to-end co-design paradigm. 

Most existing co-design frameworks, however, still rely on an explicit communication channel, such as a wireless channel, that is external to and independent of the control system. By contrast, the ICoCo framework proposed in this paper enables information transmission that improves the control objective without invoking an external channel. Instead, the control plant itself serves as the communication channel.  In this sense, it represents a deeper integration of communication and control.


Indeed, the idea of implicit communication through the control plant can be traced back to Witsenhausen's counterexample \cite{witsenhausen1968counterexample} introduced in 1968. In this seminal work, Witsenhausen formulated a decentralized LQG control problem with asymmetric information between two controllers and showed the surprising result that linear policies are not optimal. Rather, the first controller can apply a nonlinear policy that helps the second controller better infer the system state from its noisy observation, thereby reducing the overall cost. Although not explicitly described in Witsenhausen's original paper, this phenomenon was later interpreted as implicit communication from controller~1 to controller~2 \cite{grover2010implicit,Basar1987Team,H01979Teams,Mengyuan2025}. Subsequent studies have shown that, even for this simple problem, designing policies that properly balance control and communication remains highly nontrivial \cite{grover2010actions,Grover2013TAC,ouyang2022signaling,mengyuan2024}.

Information patterns are fundamental in decentralized control, as they specify what information is available to each controller at each time and therefore determine what information must be exchanged to optimize the control objective. Unfortunately, for many information patterns, the information that must be transmitted is itself unclear or implicit, as exemplified by Witsenhausen’s counterexample. As a result, decentralized control can remain highly challenging even in the presence of an explicit communication channel, when that channel is constrained by limited rate~\cite{grover2010implicit}, noise~\cite{farhadi2011suboptimal}, or delay~\cite{Matni2013}.  In this paper, we focus on a special information pattern for which the information to be transmitted is explicit. This leads to an interesting and more tractable class of decentralized control problems than the general case, where both the relevant information and the communication channel are implicit.


\section{Problem Statement} \label{sec:problem}
\subsection{System Model}
Consider the following discrete-time linear system, jointly controlled by two agents\footnote{For ease of exposition, our treatment is presented for two control agents. The extension to general cases is straightforward.}:
\begin{align}  \label{eq:system-0}
	x_{t+1} = Ax_t + B_1 v_{t} + B_2 q_{t} + w_t , \ t=0,1,2,\cdots,
\end{align}
where $x_t\in \mathbb{R}^{d_0}$ is the system state, $v_{t}\in \mathbb{R}^{d_1}$ and $q_{t}\in \mathbb{R}^{d_2}$ are the control inputs of the two agents, respectively, and $w_t\in \mathbb{R}^{d_0}$ is a zero-mean Gaussian noise with covariance $W\succ 0$. The initial state $x_0$ is a random variable, following a Gaussian distribution $\mathcal{N}(0,X_0)$. We refer to the agent controlling $v_t$ as the \textit{leader} and the agent controlling $q_t$ as the \textit{follower}. Both agents are assumed to have perfect, noise-free observations of the system state $x_t$ at each time step. 

The goal is to control the system toward a given target state $x_*$. \textit{The central assumption} of this problem is as follows: the target state, which follows a Gaussian distribution $\mathcal{N}(0,\Sigma_0)$, is revealed only to the leader at the initial time $t=0$. Moreover, there is no dedicated communication channel between the leader and the follower, meaning that the follower receives no explicit information about the target state throughout the time horizon. The objective is to design optimal control policies for both agents to minimize the following quadratic cost over a time horizon of length $n$:
\begin{align} \label{eq:Jn}
	\min_{\{{v}_t,q_t\}} \ J_n\triangleq \mathbb{E}\left[ \sum_{t=0}^{n-1} \left(z_t^\top Fz_t + {v}^\top_tG_1{v}_t + {q}^\top_tG_2{q}_t \right) + z_n^\top F_nz_n \right],
\end{align}
where $F,G_1,G_2$, and $F_n$ are predefined positive semi-definite (PSD) matrices, $z_t\triangleq x_t - x_*$.

Since both agents have noiseless observations of the system state, we can define the following augmented quantities:
\begin{align*}
	 B = \begin{bmatrix}
		B_1 &B_2
	\end{bmatrix}, {u}_t = \begin{bmatrix}
		v_{t} \\
		q_{t}
	\end{bmatrix}.
\end{align*}
Then the system defined in \eqref{eq:system-0} is equivalent to the following single-agent system:
\begin{align}  \label{eq:system-1}
	x_{t+1} = Ax_t + B {u}_t + w_t.  
\end{align}
We make the usual assumption that $(A,B)$ is controllable. Accordingly, the objective \eqref{eq:Jn} can be written as:
\begin{align} 
	\min_{\{{u}_t\}} \ J_n\triangleq \mathbb{E}\left[ \sum_{t=0}^{n-1} \left(z_t^\top Fz_t + {u}^\top_tG{u}_t\right) + z_n^\top F_n z_n \right],
\end{align}
 where $G=\mathtt{diag}(G_1,G_2)$ is a block diagonal matrix, with two PSD blocks on the diagonal: $G_1\in \mathbb{R}^{d_1 \times d_1}$ and $G_2\in \mathbb{R}^{d_2 \times d_2}$. Throughout the remainder of this paper, we interchangeably use the two-agent system representation \eqref{eq:system-0} and the single-agent system representation \eqref{eq:system-1},  depending on which is more convenient for analysis or presentation.

If the target state $x_*$ is known to both agents at the initial time $t=0$, then the problem reduces to an ordinary LQG control problem, and the optimal policy is simply given by
\begin{align*}
	{u}_t =  -K_tx_t + D_t x_*,
\end{align*}
where $K_t$ is the feedback gain, and $D_t$ is a constant matrix independent of $x_t$, given by (see, e.g., \cite{anderson2007optimal})
\begin{align*}
	K_t &= (G+B^\top \Phi_{t+1}B)^{-1}B^\top \Phi_{t+1} A, \\
	D_t &= (G+B^\top \Phi_{t+1}B)^{-1}B^\top \bar{D}_t.
\end{align*}
Here, $\Phi_{n} = \bar{D}_n = F_n$. For each $0\le t \le n-1$, $\Phi_t$ is determined by the Riccati recursion:
\begin{align*}
	\Phi_t = F + A^\top \Phi_{t+1}A - A^\top \Phi_{t+1}B(G+B^\top \Phi_{t+1}B)^{-1}B^\top \Phi_{t+1}A,
\end{align*} 
while $\bar{D}_t$ is determined recursively by $\bar{D}_t = (A - BK_t)^\top \bar{D}_{t+1} + F$.

Given a linear policy for the single-agent system of the from $u_t=-K_tx_t + D_tx_*$, the corresponding control inputs $v_{t}$ and $q_{t}$ for the two-agent system can be obtained by decomposing the feedback gain $K_t$ and the offset gain $D_t$ as follows:
\begin{align*}
	{u}_t = \begin{bmatrix}
		v_{t} \\
		q_{t}
	\end{bmatrix} = -K_t x_t + D_tx_* = - \begin{bmatrix}
		K^l_{t} \\
		K^f_{t}
	\end{bmatrix} x_t + \begin{bmatrix}
		D^l_{t} \\
		D^f_{t}
	\end{bmatrix} x_*= \begin{bmatrix}
		-K^{l}_{t} x_t + D^l_{t} x_* \\
		-K^f_{t} x_t + D^f_{t} x_*
	\end{bmatrix}.
\end{align*}
Note that $K_t$ and $D_t$ are invariant to the target state $x_*$.
Hence, in our problem---where $x_*$ is unknown to the follower---the key challenge in this control task is that the follower can not compute its optimal offset $D^f_t x_*$. This naturally raises a question: \textit{is it possible for the follower to estimate $x_*$ from the observed state sequence?} In this paper, we affirmatively answer this question by proposing an implicit communication scheme that enables the leader to convey information about $x_*$ through its control inputs.

\subsection{Coordination via Implicit Communication}
The key idea behind implicit communication \cite{chen2025implicitcommunication} is that the leader encodes information about the target state $x_*$ into its control inputs by deliberately deviating from its optimal control policy. This enables the follower to gradually refine its estimate of $x_*$ based on the observed state trajectory. In this section, we illustrate the fundamental concept of implicit communication by showing how the control plant itself can serve as a communication channel from the leader to the follower. The detailed communication scheme will be presented in the following sections. 

We begin by defining the implicit communication channel from the leader to the follower. In particular, the leader's control input $v_t$ acts as the channel input at time $t$, and the corresponding channel output is defined as
\begin{align}  \label{eq:def-yt1}
	y'_t \triangleq x_{t+1} - Ax_t -B_2 q_t.
\end{align}
This output $y'_t$ can be computed by the follower upon observing the new state $x_{t+1}$. This construction induces a Gaussian MIMO channel~\cite{gallager1968information} from the leader (the transmitter) to the follower (the receiver), characterized by
\begin{align}  \label{eq:yt1}
	y'_t = B_1 v_t + w_t.
\end{align}
Here, $v_t$ is the channel input, $y'_t$ is the channel output, and $w_t$ is the channel noise. The resulting implicit communication channel is illustrated in Fig.~\ref{fig:channel}.

\begin{figure}[t]
	\centering
	\includegraphics[width=0.8\textwidth]{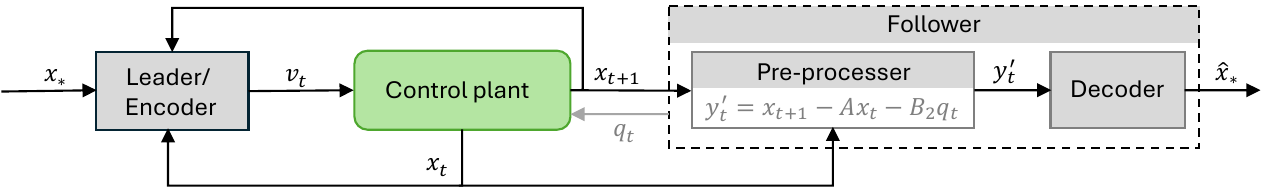}
	\caption{The control plant as a communication channel from the leader to the follower.} 
	\label{fig:channel}
\end{figure}

Communication over Gaussian MIMO channels is a well-studied problem, and many channel coding schemes are known to be asymptotically optimal in terms of communication rate~\cite{richardson2008modern}. The key challenge that distinguishes our implicit communication problem from classical communication settings is that the channel input $v_t$ is not solely a communication signal, but also a control input. Consequently, $v_t$ must be designed in a way that simultaneously accounts for control performance and communication efficiency.
This makes implicit communication an inherently integrated communication and control problem, which is significantly more challenging than classical communication problems.

Fortunately, as shown in \cite{chen2025implicitcommunication}, implicit communication in LQG control systems with noiseless observations at both the transmitter and the receiver admits a \textit{separation principle}: the optimal channel input policy that balances communication and control can be decomposed into independent control and communication components. Specifically, recall that the transmitter’s optimal control input is given by $v^*_t = -K^{l}_{t} x_t + D^l_t x_*$, then the optimal channel input takes the form 
\begin{align} \label{eq:channel-vt}
	v_t = v^*_t + s_t =  -K^{l}_{t} x_t + D^l_t x_*+ s_t,
\end{align}
where $s_t$ is a zero-mean Gaussian communication signal, independent of $v^*_t$. Under this structure, the message to be transmitted is encoded entirely into $s_t$. From a control perspective, $s_t$ can be viewed as an artificial disturbance that degrades control performance. Therefore, the fundamental tradeoff between control and communication is governed by the power of the communication signal, characterized by the covariance matrix of $s_t$:  increasing this power improves the signal-to-noise ratio (SNR) and thereby enhances communication performance, but at the expense of control performance.

Motivated by the separation principle, it is natural for the leader to transmit the target state $x_*$ to the follower using an input of the form \eqref{eq:channel-vt}.
However, the implicit communication problem studied in \cite{chen2025implicitcommunication} assumes that the message to be transmitted is a uniformly distributed  and discrete random variable that is independent of the control task. This assumption ensures that the communication signal $s_t$ is independent of the transmitter's optimal control input $v^*_t$. Our setting differs in two important aspects: First, the message $x_*$ follows a Gaussian distribution rather than a uniform one. Second, $x_*$ is not independent of the control task; instead, it directly influences the leader's optimal control input through the offset term $D^l_t x_*$. These differences change the structure of the implicit communication problem and require a tailored communication strategy.

The first difference implies that source coding (i.e., compression) may be needed to handle the non-uniform distribution of the message, transforming the problem into one of joint source-channel coding (JSCC). Fortunately, as we will show in the following sections, when both the message and the channel noise are Gaussian, this JSCC problem admits a relatively simple and elegant solution.
 To address the second challenge, we propose replacing the leader's offset term $D^l_t x_*$ with a sub-optimal offset $c^l_t$, which is constructed in a way that is known to both the leader and the follower. As a result, we redefine the channel output using all information available at the follower:
\begin{align} \label{eq:output-yt}
	y_t \triangleq y'_t + B_1(K_t^lx_t - c^l_t) = B_1s_t + w_t. 
\end{align}

In summary, to enable implicit coordination between the two agents for optimizing the control task, we consider control inputs of the form
\begin{align} \label{eq:inputs-1}
	v_t = -K^{l}_{t} x_t + c^l_t + s_t, \ q_t = -K_t^f x_t + c^f_t,
\end{align}
where $c_t = [c^l_t, c^f_t]$ is determined by common information available to both the leader and the follower. By encoding information about the target state $x_*$ into the communication signal $s_t$, the follower can estimate $x_*$ from the output sequence $\{y_t\}$. As the follower's estimation error decreases over time, the offset $c_t$ converges toward its optimal value, enabling the two agents to effectively coordinate in steering the system toward the target state. The goal is to design a policy that determines the communication signal $s_t$  so as to minimize the overall control cost defined in \eqref{eq:Jn}.

Before delving into the design and optimization of the implicit communication scheme, we first present the following definition, which is important from both the control and communication perspectives.
\begin{definition}
	Consider the linear system \eqref{eq:system-0}. The leader is said to be \textit{fully actuated} if $\rk(B_1)= d_0$; otherwise, it is called \textit{under-actuated}.
\end{definition}

From a control perspective, $\rk(B_1)= d_0$ implies that $(A,B_1)$ is controllable, meaning the leader alone can drive the system to any desired target state without assistance from the follower. In contrast, if $\rk(B_1) < d_0$, the leader lacks full control authority, and achieving a general target state typically requires coordinated effort from both agents. That said, even when the leader is fully actuated, coordination with the follower may still reduce the overall control cost. This will be verified in our experiments.

From a communication perspective, $B_1$ is the channel gain of the implicit channel defined in \eqref{eq:output-yt}. If the rank of $B_1$ matches the dimension of the message $x_*$, the communication scheme can be relatively simple and elegant; otherwise, the channel becomes rank-deficient, and not all dimensions of the message can be transmitted directly within a single channel use. In this case, a more complicated communication scheme is needed to handle the dimensional mismatch. 

In the following sections, we first propose an implicit coordination scheme for the fully actuated setting, and then extend it to handle the under-actuated setting.

\section{Fully Actuated Leader} \label{sec:full}
This section presents an algorithm for coordinating the two agents via implicit communication in the fully actuated setting. Under this algorithm, the follower progressively refines its estimate of the target state $x_*$ from the observed state trajectory. We prove that the estimation error at the follower decreases over time and converges to $0$ as time $t\to \infty$, provided that the communication signal maintains a positive definite covariance matrix at every time step. Building on this result, we optimize the communication power to minimize the overall control cost. Under a mild assumption on the terminal estimation error covariance, we derive necessary conditions for the optimal communication power allocation.

\subsection{Implicit Coordination Scheme} \label{subsec:coding}
We consider control inputs of the form given in \eqref{eq:inputs-1}, which reduces to the optimal control inputs if $s_t=0$ and $c_t = D_t x_*$. Since the follower does not have access to the target state $x_*$, it is natural to set $c_t^f = D_t^f x_*^{(t)}$, where $x_*^{(t)}$ denotes the follower's estimate of  $x_*$ at the beginning of time $t$ (i.e., the end of time $t-1$ in continuous time). The objective is to design a scheme that determines both $s_t$ and $x_*^{(t)}$ such that $s_t$ goes to zero and $x_*^{(t)}$ converges to $x_*$ as time progresses. 

Since $x_*^{(t)}$ must be a function of the state sequence $x_0,x_1,\dots,x_{t}$, the leader can also compute $x_*^{(t)}$ from the state sequence. Consequently, the channel output $y_t$ is available to both the leader and the follower at each time step, meaning that the channel \eqref{eq:output-yt} operates with perfect feedback \cite{Sanjoy2009feedback}. Our method is inspired by Elias' scheme  \cite{elias1956channel,Gallager-SKscheme}, which was originally proposed for transmitting a scalar Gaussian source over a scalar Gaussian channel with feedback.

Note that $s_t\in \mathbb{R}^{d_1}$, $x_*\in \mathbb{R}^{d_0}$, and $B_1\in \mathbb{R}^{d_0\times d_1}$. By definition, the leader is fully actuated if $\rk(B_1)=d_0$, implying that $d_1\ge d_0$. At the initial time step $t=0$, we set the communication signal $s_0$ as a linear function of the message $x_*$:
\begin{align*}
	s_0 =Q S_0^{\frac{1}{2}} \Sigma_{0}^{-\frac{1}{2}} x_*,
\end{align*}
where $S_0\in \mathbb{R}^{d_0 \times d_0}$ is a PD matrix and $Q\in \mathbb{R}^{d_1\times d_0}$ is a projection matrix that maps a $d_0$-dimensional signal into a $d_1$-dimensional space. To ensure that this linear transformation preserves all information about $x_*$, we require $Q$ to have full rank. As a result, $Q$ admits a Moore–Penrose inverse $Q^\dagger = (Q^\top Q)^{-1} Q^\top$, which satisfies $Q^\dagger Q=I$. In practice, a simple choice of $Q$ is  $Q=I$ when $d_0=d_1$ and $Q=B_1^\top$ when $d_0<d_1$. Under this construction, the covariance of $s_0$ is given by
\begin{align*}
	\cov(s_0) = Q S_0 Q^\top.
\end{align*}

Now that the communication signal has been determined, the next step is to define the control inputs of the two agents. At the initial time step $t=0$, the follower has no prior information about $x_*$ beyond knowing that it is a zero-mean Gaussian random variable. Therefore, the follower's estimate of $x_*$ at time $t=0$ is $x_*^{(0)}=\mathbb{E}[x_*]=0$. Based on this estimate, we set the follower's control input as:
\begin{align*}
	q_0 = -K_0^fx_0 + D_0^f x_*^{(0)} = -K_0^fx_0. 
\end{align*}
As for the leader's offset term, instead of applying the optimal value $D_0^l x_*$, we propose using $c_0^l = D_0^l x_*^{(0)}=0$, which can also be computed by the follower. Therefore, the leader's control input is given by
\begin{align*}
	v_0 = -K_0^l x_0 + s_0.
\end{align*}
The motivation for this choice will be discussed later (see Remark 1). Now, upon applying $v_0$ and $q_0$, the system transitions to a new state $x_1$. As discussed around \eqref{eq:output-yt}, both agents can compute the implicit channel output as:
\begin{align*}
	y_0  = x_1 - (A-BK_0)x_0 = B_1s_0 + w_0.
\end{align*}
This channel output $y_t$ enables the follower to estimate $x_*$. In particular, the minimum mean-square error (MMSE) estimate of $x_*$ is given by
\begin{align*}
	 \hat{e}_0 \triangleq \mathbb{E}[x_*|y_0] = \Sigma_{0}^{\frac{1}{2}}S_0^{-\frac{1}{2}} Q^\dagger \mathbb{E}[s_0|{y}_0] .
\end{align*}
Since $w_0$ is Gaussian and independent of $s_0$, we have
\begin{align*}
	\mathbb{E}[s_0|{y}_0] & = \mathbb{E}[s_0] + \cov(s_0, {y}_0)B_1^\top \cov({y}_0)^{-1}({y}_0 - \mathbb{E}[{y}_0]) \\
	& = \cov(s_0)B_1^\top [B_1\cov(s_0)B_1^\top  + \cov({w}_0)]^{-1}{y}_0.
\end{align*}
It follows that the follower's estimate of $x_*$ at the beginning of time $t=1$ is
\begin{align} \label{eq:hate0}
	x_*^{(1)} = \hat{e}_{0} =  \Sigma_{0}^{\frac{1}{2}} S_0^{\frac{1}{2}} Q^\top B_1^\top (B_1QS_0 Q^\top B_1^\top + W)^{-1}  y_0.
\end{align}
Since the leader can also compute $y_0$ using $x_0$ and $x_1$, it knows exactly the follower's estimate of $x_*$ at the beginning of time $t=1$. The leader can then compute the follower's estimation error as:
\begin{align*}
	e_1 \triangleq x_* - \hat{e}_0.
\end{align*}
It is easy to verify that $e_1$ is a zero-mean Gaussian random variable. We defer the derivation of its covariance, denoted by $\Sigma_1 = \cov(e_1)$, to Lemma~\ref{lem:estimation}.

At the next time step (i.e., $t=1$), the leader treats $e_1$ as the new message and aims to transmit $e_1$ to refine the follower's estimate of $x_*$. Following the same procedure as before, the follower computes $y_1$ from $x_2$ and $x_1$, and then forms the MMSE estimate $\hat{e}_1 = \mathbb{E}[e_1|y_1]$. Based on this result,  the follower's estimate of $x_*$ can be updated accordingly:
\begin{align*}
	x^{(2)}_* = \hat{e}_0 + \hat{e}_1.
\end{align*}
At the beginning of time $t=2$, the leader again computes the last step's estimation error $e_2 = e_1 - \hat{e}_1$, treats it as the new message, and transmits it to further refine the follower's estimate. Note that $e_2$ corresponds to the follower's current estimation error of $x_*$:
\begin{align*}
	e_2 = x_* - \hat{e}_0 - \hat{e}_1 = x_* - x^{(2)}_*.
\end{align*}
This procedure is repeated until the last control time $t=n-1$. 

In general, for $1\le t\le n-1$, the communication signal is constructed as a linear function of the message $e_t\sim \mathcal{N}(0,\Sigma_t)$:
\begin{align}
	s_t =Q S_t^{\frac{1}{2}} \Sigma_{t}^{-\frac{1}{2}} e_{t},
\end{align}
where $Q$ is a predefined projection matrix, and $S_t$ is a PD matrix that determines the power of the communication signal and will be optimized.
The two agents then compute their control inputs based on the follower's current estimate $x^{(t)}_*$:
\begin{align}
	v_t = -K_t^l x_t + D_t^l x^{(t)}_* + s_t, \ q_t = -K_t^fx_t + D_t^f x_*^{(t)}. 
\end{align}
Upon observing $x_{t+1}$, the implicit channel output $y_t$ is computed as
\begin{align}
	y_t = x_{t+1} - (A-BK_t)x_t - BD_t x_*^{(t)}.
\end{align}
Using a similar argument as in \eqref{eq:hate0}, the MMSE estimate of $e_t$ given $y_t$ is 
\begin{align} \label{eq:hat-et}
	\hat{e}_{t} = \mathbb{E}[e_t|y_t] =  \Sigma_{t}^{\frac{1}{2}} S_t^{\frac{1}{2}} Q^\top B_1^\top (B_1QS_t Q^\top B_1^\top + W)^{-1}  y_t.
\end{align}
At the beginning of time $t+1$, the follower updates its estimate of the target state $x_*$: 
\begin{align}
	x^{(t+1)}_* = x^{(t)}_* + \hat{e}_t = \sum_{i=0}^{t} \hat{e}_i.
\end{align}
The corresponding estimation error is updated as $e_{t+1} = x_* - x^{(t+1)}_* = e_t - \hat{e}_t$, which becomes the message to be transmitted at time $t+1$. It can be verified that $e_{t+1}$ remains a zero-mean Gaussian random variable if $e_t$ is zero-mean Gaussian. The complete procedure is summarized in Algorithm~\ref{alg:1}, referred to as implicit coordination scheme.

\begin{algorithm}[!t]
	\caption{Implicit Coordination Scheme (fully actuated leader)}
	\label{alg:1}
	\SetAlgoLined
	\KwIn{target state $x_*$, target state covariance $\Sigma_0$, horizon length $n$, matrix sequence $\{S_t\}$}
	
	\textbf{Initialization:} $c_0 = [c_0^l, c_0^f] = 0, e_{0} = x_*, x_*^{(0)}=0$, projection matrix $Q$
	
	\For{$t = 0$ \KwTo $n-1$}{
		
		Compute the communication signal:\
		\[ 
		s_t =Q S_t^{\frac{1}{2}} \Sigma_{t}^{-\frac{1}{2}} e_{t}
		\] \\
		Compute the offset:
		$$c_{t}=[c^l_{t}, c^f_{t}] = D_{t}x_*^{(t)}$$ \\
		Compute the control inputs:\
		\[ 
		v_t =-K^l_t x_t  + c^l_t + s_t, \ q_t = -K_t^f x_t + c^f_t
		\] \\
		The leader inputs $v_t$  and the follower inputs $q_t$; both agents observe the new state $x_{t+1}$ \\
		Compute the channel output:\
		$$y_{t}=x_{t+1}- (A-BK_t)x_t - Bc_t$$ \\
		Estimate $e_t$: 
		$$\hat{e}_{t} =  \Sigma_{t}^{\frac{1}{2}} S_t^{\frac{1}{2}} Q^\top B_1^\top (B_1QS_t Q^\top B_1^\top + W)^{-1}  y_t$$ \\
		Update the estimation error and the corresponding error covariance:\
		$$e_{t+1} = e_t - \hat{e}_t, \ \Sigma_{t+1} = \Sigma_{t}^{\frac{1}{2}} V_t \Sigma_{t}^{\frac{1}{2}}$$ \\
		Update the target state estimate $x_*^{(t+1)} = x_*^{(t)} + \hat{e}_{t} $.

	}
	
\end{algorithm}

To implement Algorithm~\ref{alg:1}, we need to compute the estimation error covariance $\Sigma_t=\cov(e_t)$ at each time step, as it is required both for constructing the communication signal $s_t$ and computing the estimate $\hat{e}_t$. In principle, $S_t$ can be chosen as any PD matrix. Given this choice and using equation \eqref{eq:hat-et}, $\Sigma_{t+1}$ can be derived as a function of $\Sigma_t$ and $S_t$. However, allowing a general form of $S_t$ leads to a complicated expression for $\Sigma_{t+1}$, which poses challenges not only for computational efficiency but also for subsequent analysis and optimization. This motivates us to restrict $S_t$ to a more structured form that simplifies the expression of $\Sigma_{t+1}$.

Since $Q$ is a predefined matrix, the implicit communication channel \eqref{eq:output-yt} can be equivalently written as
\begin{align} \label{eq:yt-eqv}
	y_t = B_1Q \tilde{s}_t + w_t,
\end{align}
where $\tilde{s}_t = S_t^{\frac{1}{2}} \Sigma_{t}^{-\frac{1}{2}} e_{t}$ is the input to this equivalent channel. Then it is easy to see that $S_t$ corresponds to the covariance of $\tilde{s}_t$. The channel defined in \eqref{eq:yt-eqv} is a Gaussian MIMO channel with channel gain matrix $B_1Q$, it is known in information theory that the capacity-achieving input covariance $S_t$ is simultaneously unitarily diagonalizable with $(B_1Q)^\top W^{-1}B_1Q$ (see, e.g., \cite{telatar1999capacity}). Although our goal is to minimize the overall control cost rather than maximize the communication rate, adopting this structure remains a natural and effective choice. In practice, it leads to a greatly simplified expression for the estimation error covariance, which is crucial for both analysis and optimization.

Specifically, since $W$ is PD, the matrix $(B_1Q)^\top W^{-1}B_1Q$ must be PSD and thus admit a diagonal decomposition. Without loss of generality, assume $(B_1Q)^\top W^{-1}B_1Q=UH U^\top$, where $U$ is a unitary matrix and $H$ is a diagonal matrix. We restrict $S_t$ to have the form $S_t = U \Lambda_t U^\top$, where $\Lambda_t$ is a diagonal matrix. This structural assumption allows us to compute the estimation error covariance in a simple, recursive manner, as established in the following lemma.

\begin{lemma}  \label{lem:estimation}
	Given $(B_1Q)^\top W^{-1}B_1Q=UH U^\top$ and $S_t = U\Lambda_t U^\top$. Let $\Lambda_t(i)$ and $H(i)$ denote the $i$-th diagonal element of $\Lambda_t$ and $H$, respectively. Then the estimation error process $\{e_t:t\ge 0\}$ evolves according to the Markov recursion
	\begin{align*}
		e_{t+1} = \Sigma_{t}^{\frac{1}{2}} V_t  \Sigma_{t}^{-\frac{1}{2}} e_{t} - \Sigma_{t}^{\frac{1}{2}} V_t S_t^{\frac{1}{2}} Q^\top B_1^\top  W^{-1} w_t,
	\end{align*}
	where $V_t = U\hat{V}_t U^\top$, and $\hat{V}_t = (I + \Lambda_tH)^{-1}$ is a diagonal matrix, with the $i$-th diagonal entry given by
	\begin{align*}
		\hat{V}_t(i) = \frac{1}{1 + \Lambda_t(i) H(i)}.
	\end{align*}
	Moreover, $e_{t+1}\sim \mathcal{N}(0,\Sigma_{t+1})$, where $\Sigma_{t+1} = \Sigma_{t}^{\frac{1}{2}} V_t \Sigma_{t}^{\frac{1}{2}} $.
\end{lemma}
\begin{IEEEproof}
	See Appendix.
\end{IEEEproof}

The significance of Lemma~\ref{lem:estimation} lies not only in providing a computationally efficient implementation of Algorithm~\ref{alg:1}, but also in enabling a theoretical performance guarantee. In particular, the recursive expression for the estimation error covariance allows us to prove that if all the diagonal entries of $\Lambda_t$ are positive for every time step $t$, then the follower’s estimate of $x_*$ becomes increasingly accurate and ultimately converges exactly to $x_*$.

\begin{theorem} \label{thm:covet}
	Let  $\Sigma_t$ denote the follower's estimation error covariance of $x_*$ at time $t$ under Algorithm~\ref{alg:1}. For any fixed projection matrix $Q$, suppose there is a constant $\sigma>0$ such that $\Lambda_t(i)\ge \sigma$ for all $t$ and all $i$. Then, for any $t\ge 1$, it holds that $\Sigma_t \succeq \Sigma_{t+1}$. Moreover, if $ \rk(B_1 Q) =\rk(B_1)= d_0$,
	then $\Sigma_t \succ \Sigma_{t+1}$ and 
	\begin{align*}
		\tr(\Sigma_{t}) \le \frac{1}{(1+\sigma \psi)^{t}} \tr({\Sigma}_0),
	\end{align*}
	where $\psi = \min_{i} H(i)$.
\end{theorem}
\begin{IEEEproof}
	According to Lemma 1,
	\begin{align*}
		\Sigma_{t+1} = \Sigma_{t}^{\frac{1}{2}} V_t \Sigma_{t}^{\frac{1}{2}} = \Sigma_{t+1} = \Sigma_{t}^{\frac{1}{2}} U \hat{V}_t U^\top \Sigma_{t}^{\frac{1}{2}} 
	\end{align*}
	It follows that
	\begin{align*}
		\Sigma_t - \Sigma_{t+1} = \Sigma_{t}^{\frac{1}{2}} U (I - \hat{V}_t ) U^\top \Sigma_{t}^{\frac{1}{2}} 
	\end{align*}
	Recall that $\hat{V}_t$ is a diagonal matrix whose $i$-th diagonal entry is given by 
	\begin{align*}
		\hat{V}_t(i) = \frac{1}{1 + \Lambda_t(i) H(i)}.
	\end{align*}
	Since $(B_1Q)^\top W^{-1}B_1Q$ is PSD, its eigenvalues $H(i)\ge 0$ for all $i$.
	It is then clear that, if all the diagonal entries of $\Lambda_t$ are non-negative, then all the diagonal entries of $I-\hat{V}_t$ are also  non-negative. As a result, $\Sigma_t -\Sigma_{t+1} \succeq 0$. 
	
	If $\rk(B_1)=d_0$ and $B_1 Q \in \mathbb{R}^{d_0\times d_0}$ is full rank (e.g., let $Q=B_1^\top$). Consequently, $ (B_1 Q)^\top W^{-1} B_1Q$ is PD, meaning that $H(i)>0$ for all $i$. In this case, we have $\Sigma_t -\Sigma_{t+1} \succ 0$ for all $t$. To show that $\tr(\Sigma_t)$ converges to zero, let $\bar{\Sigma}_t = U^\top \Sigma_t U$ and $\psi = \min_{i} H(i)$. Then we have
	\begin{align*}
		\tr(\Sigma_{t+1}) & = \tr(\Sigma_{t}^{\frac{1}{2}} V_t \Sigma_{t}^{\frac{1}{2}} ) = \tr(\hat{V}_t U^\top \Sigma_t U) \\
		& = \sum_{i=1}^{d_0} \hat{V}_t(i) \bar{\Sigma}_t(i,i) \\
		& \le \frac{1}{1+\sigma \psi} \sum_{i=1}^{d_0} \bar{\Sigma}_t(i,i) \\
		& = \frac{1}{1+\sigma \psi} \tr(\bar{\Sigma}_t) \\
		& = \frac{1}{1+\sigma \psi} \tr({\Sigma}_t)  \\
		& \le \frac{1}{(1+\sigma \psi)^{t+1}} \tr({\Sigma}_0).
	\end{align*}
This completes the proof.
\end{IEEEproof}

Theorem \ref{thm:covet} highlights the importance of choosing a full-rank projection matrix $Q$; otherwise, the estimation error covariance may not decrease strictly over time. When $Q$ is full-rank, and since $\sigma \psi >0$ and $\tr(\Sigma_t)\ge 0$ for all $t$, it follows immediately from Theorem~\ref{thm:covet} that $\tr(\Sigma_{t})\to 0$ as $t\to \infty$. In this case, the theorem implies that if $\Lambda_t$ is chosen to have strictly positive diagonal entries for all $t$ and gradually decreases toward the zero matrix over time, then both agents' control input will asymptotically converge to the optimal ones. As a result, the system can achieve the target state $x_*$ provided that the time horizon is sufficiently long. The remaining task is to optimize the sequence $\{\Lambda_t\}$ over time to minimize the overall control cost.

\begin{remark}
	Before proceeding to the optimization of communication power, we justify the choice of the leader's offset term in Algorithm~\ref{alg:1}. Although the leader knows the target state $x_*$, and it may seem natural to set its offset term as the optimal value $D_t^lx_*$; we instead choose $c^l_t=D_t^l x_*^{(t)}$. Why prefer this sub-optimal offset? To see the motivation, suppose we use $c^l_t = D_t^lx_*$. Then the system evolves as
	\begin{align*}
		x_{t+1} & = (A-BK_t)x_t + B_1 D^l_t x_* + B_2 D^f_t x_*^{(t)} + B_1 s_t + w_t, \\
		&= (A-BK_t)x_t + B_1 D^l_t (x_*^{(t)} + e_t) + B_2 D^f_t x_*^{(t)} + B_1 s_t + w_t. 
	\end{align*}
	The follower can still construct the implicit channel output as $y_t = x_{t+1} - (A-BK_t)x_t - BD_t x_*^{(t)}$. However, the implicit channel now is defined as
	\begin{align} \label{eq: yt-opt-offset}
		y_t = B_1 D^l_t e_t + B_1 s_t + w_t.
	\end{align}
	By defining $s_t$ as a linear function of $e_t$, the modified channel still allows $e_t$ to be estimated from $y_t$. However, the expression of the estimation error covariance becomes significantly more complicated---even when $S_t$ is restricted to the form $S_t = U\Lambda_t U^\top$. This added complexity not only increases the computational burden of the algorithm but also poses substantial challenges for optimizing the communication power. Moreover, in the modified channel \eqref{eq: yt-opt-offset}, $(D^l_t e_t + s_t)$ can be viewed as the communication signal, yet only the power of $s_t$ can be optimized. From the perspective of control performance, it is not clear that using the nominally optimal offset yields better results. For these reasons, we adopt the cleaner formulation $c^l_t=D_t^l x_*^{(t)}$, which leads to a more elegant and tractable analysis.
	
\end{remark}

\subsection{Communication Power Optimization}
This section addresses the final step in implementing  Algorithm~\ref{alg:1}---the optimization of communication power, governed by  the sequence of matrices $\{\Lambda_t\}$.
With the aim of minimizing the overall control cost $J_n$, we formulate this optimization problem as a deterministic MDP with matrix-valued states and actions. Under a mild condition on the terminal estimation error covariance, the optimal solution to this deterministic MDP can be derived using the discrete matrix minimum principle \cite{athans1966discrete,LQG-delay-1974}.

We first derive the control cost under Algorithm~\ref{alg:1}, expressing it as a function of the decision variables $\{\Lambda_t\}$. Recall that, under Algorithm~\ref{alg:1}, the joint control input at time $t$ is given by
\begin{align} \label{eq: ut}
	u_t = -K_t x_t + D_t x_*^{(t)} + \tilde{I}{s}_t& = -K_t x_t + D_t(x_* - e_t) +  \tilde{I}{s}_t  \notag \\
	& = -K_t(x_t -x_*) + c'_t + (\tilde{I} QS_t^{\frac{1}{2}} \Sigma_{t}^{-\frac{1}{2}}  - D_t )e_t ,
\end{align}
where $c'_t\triangleq (D_t-K_t)x_*$ is a constant, and $\tilde{I} \triangleq [\mathbf{I} \ \mathbf{0}]^\top$, in which $\mathbf{I}$ is the $d_1\times d_1$ identity matrix and $\mathbf{0}$ is the $d_1\times d_2$ zero matrix. To simplify notations, let $\bar{A}_t \triangleq A-BK_t$ and $Q_1 \triangleq B_1Q$.
Then the system state evolves as follows:
\begin{align*}
	x_{t+1} = \bar{A}_t x_t + B D_tx_* + (Q_1 S_t^{\frac{1}{2}} \Sigma_{t}^{-\frac{1}{2}} - BD_t )e_{t} + w_t. 
\end{align*}
Let $Z_t \triangleq \cov(z_t)$, where $z_t = x_t - x_*$ is the state error at time $t$. It follows immediately that
\begin{align} \label{eq:error-zt}
	z_{t+1} = x_{t+1} - x_* =  \bar{A}_tz_t +  c''_t + (Q_1S_t^{\frac{1}{2}} \Sigma_{t}^{-\frac{1}{2}} -BD_t )e_{t} + w_t,
\end{align}
where $c''_t \triangleq (A-\mathbf{I})x_* + B c'_t $ is a constant.
Denote by $\Omega_t=\cov(z_t,e_t)$ the covariance between $z_t$ and $e_t$. 
Then according to \eqref{eq: ut}, the expected input cost at each time step is given by
\begin{align*}
	\mathbb{E}[{u}_t^\top G {u}_t] =& \mathbb{E}[ [-K_tz_t + c'_t + (\tilde{I} QS_t^{\frac{1}{2}} \Sigma_{t}^{-\frac{1}{2}}  - D_t )e_t ]^\top G [-K_tz_t + c'_t + (\tilde{I}QS_t^{\frac{1}{2}} \Sigma_{t}^{-\frac{1}{2}}  - D_t )e_t ] ]\\
	= &  \tr(Z_t K_t^\top G K_t) + \tr(Q^\top G_1 Q S_t) + \tr(D_t^\top G D_t \Sigma_t) - 2\tr(D_t^\top G \tilde{I}QS_t^{\frac{1}{2}} \Sigma_{t}^{\frac{1}{2}} ) - 2\tr((\tilde{I}QS_t^{\frac{1}{2}} \Sigma_{t}^{-\frac{1}{2}}  - D_t )^\top G K_t \Omega_t ) \\ 
	&+ \text{constant},
\end{align*}
where the constant includes terms that are independent of the decision variable $\Lambda_t$. In addition,
\begin{align*}
	\mathbb{E}[z_t^\top F z_t] =  \tr(Z_t F) + x_*^\top F x_*.
\end{align*}
Therefore, the quadratic control cost under Algorithm~\ref{alg:1} is given by
\begin{align*}
	J_n =& \sum_{t=0}^{n-1} \tr( (F+K_t^\top G K_t)Z_t) + \tr(Q^\top G_1 Q S_t) + \tr(D_t^\top G D_t \Sigma_t) - 2\tr(D_t^\top G \tilde{I}QS_t^{\frac{1}{2}} \Sigma_{t}^{\frac{1}{2}} ) - 2\tr((\tilde{I}QS_t^{\frac{1}{2}} \Sigma_{t}^{-\frac{1}{2}}  - D_t )^\top GK_t \Omega_t )  \\
	&+ \tr(Z_n F_n)  + \text{constant}.
\end{align*}
Next, we show that $J_n$ can be viewed as the cost of a deterministic MDP, where the state is a tuple of matrices and the action is $\Lambda_t$. The key step is to derive the state transition equations.

Since $z_t$ and $e_t$ are correlated, we treat them jointly by constructing an augmented state. In particular,
let $\rho_t=[z_t, e_t]$. Then according to Lemma~\ref{lem:estimation} and equation~\eqref{eq:error-zt}, the process $\{\rho_t:t\ge 0\}$  evolves according to the following equation:
\begin{align}  \label{eq: rhot}
	\rho_{t+1} = A^\rho_t \rho_t  + B^\rho_t w_t + \bar{c}_t,
\end{align}
where
\begin{align*}
	A^\rho_t \triangleq \begin{bmatrix}
		\bar{A}_t & Q_1S_t^{\frac{1}{2}} \Sigma_{t}^{-\frac{1}{2}} - BD_t \\
		0 & \Sigma_{t}^{\frac{1}{2}} V_t  \Sigma_{t}^{-\frac{1}{2}}
	\end{bmatrix}, \ B^\rho_t \triangleq \begin{bmatrix}
		I \\
		- \Sigma_{t}^{\frac{1}{2}} V_t S_t^{\frac{1}{2}} Q_1^\top  W^{-1} 
	\end{bmatrix}, \ \bar{c}_t \triangleq \begin{bmatrix}
	c''_t \\
	0
	\end{bmatrix}.
\end{align*}
According to \eqref{eq: rhot}, the covariance of $\rho_t$ evolves as follows:
\begin{align} \label{eq:cov-rhot}
	\cov(\rho_{t+1}) = A^\rho_t  \cov(\rho_t) (A^\rho_t )^\top + B^\rho_t  W (B^\rho_t )^\top.
\end{align}
Given the definition of $\rho_t$, we can express the covariance matrix of $\rho_t$ as the following block matrix:
\begin{align*}
	\cov(\rho_t) = \begin{bmatrix}
		Z_t & \Omega_t \\
		\Omega_t^\top & \Sigma_t
	\end{bmatrix}.
\end{align*}
Matching the blocks on both sides of \eqref{eq:cov-rhot} yields
\begin{align} \label{eq: cov-zt}
	Z_{t+1}  = & \bar{A}_t Z_t\bar{A}_t ^\top + (Q_1S_t^{\frac{1}{2}} \Sigma_{t}^{-\frac{1}{2}} - BD_t)\Omega^\top_t\bar{A}_t ^\top 
	+ \bar{A}_t \Omega_t (Q_1S_t^{\frac{1}{2}} \Sigma_{t}^{-\frac{1}{2}} - BD_t)^\top  + W  \notag \\
	& + Q_1S_t Q_1^\top +BD_t \Sigma_t D_t^\top B^\top - BD_t \Sigma_{t}^{\frac{1}{2}} S_{t}^{\frac{1}{2}} Q_1^\top - Q_1 S_{t}^{\frac{1}{2}} \Sigma_{t}^{\frac{1}{2}} D_t^\top B^\top,   \\  \label{eq: cov-sigmat}
	\Sigma_{t+1}  =& \Sigma_{t}^{\frac{1}{2}} V_t  \Sigma_{t}^{\frac{1}{2}},  \\  \label{eq: cov-pit}
	\Omega_{t+1}  = &  \bar{A}_t \Omega_t   \Sigma_{t}^{-\frac{1}{2}} V_t  \Sigma_{t}^{\frac{1}{2}} - B D_t\Sigma_{t}^{\frac{1}{2}} V_t  \Sigma_{t}^{\frac{1}{2}}.
\end{align}
Note that the cost $J_n$ is completely determined by $(Z_t,\Sigma_t,\Omega_t)$ and $\Lambda_t$. It is therefore natural to treat $(Z_t,\Sigma_t,\Omega_t)$ as the state of the MDP at time $t$. However, we observe that the system \eqref{eq: cov-zt}-\eqref{eq: cov-pit} can be further simplified.
In particular, let $L_t \triangleq \Omega_t \Sigma_t^{-1}$. Then, according to Lemma~\ref{lem:estimation}, \eqref{eq: cov-pit} can be written as
\begin{align} \label{eq: cov-Lt}
	L_{t+1}  = &  \bar{A}_t L_t -  B D_t.
\end{align}
With this change of variables, both the system dynamics and the cost function can be expressed in terms of $Z_t, \Sigma_t, L_t$, without explicitly involving $\Omega_t$. Moreover, \eqref{eq: cov-Lt} implies that the sequence $\{L_t\}$ is constant and independent of the input $\Lambda_t$. As a result, the system can be represented using state $(Z_t,\Sigma_t)$ and input $\Lambda_t$, as follows:
\begin{align} \label{eq: cov-zt-2}
	Z_{t+1}  =  f^Z_t(Z_t,\Sigma_t,\Lambda_t) \triangleq& \bar{A}_t Z_t\bar{A}_t ^\top + Q_1S_t Q_1^\top + Q_1S_t^{\frac{1}{2}} \Sigma_{t}^{\frac{1}{2}}L_{t+1}^\top - BD_t \Sigma_t L^\top_t \bar{A}_t ^\top
	  + W  \notag \\
	& +  (Q_1S_t^{\frac{1}{2}} \Sigma_{t}^{\frac{1}{2}}L_{t+1}^\top - BD_t \Sigma_t L^\top_t \bar{A}_t ^\top)^\top  +BD_t \Sigma_t D_t^\top B^\top,   \\  \label{eq: cov-sigmat-2}
	\Sigma_{t+1}  = f^\Lambda_t(\Sigma_t,\Lambda_t) \triangleq & \Sigma_{t}^{\frac{1}{2}} V_t  \Sigma_{t}^{\frac{1}{2}}. 
\end{align}
Define the cost function at time $t$ as follows: $l_n(Z_n, \Sigma_n, \Lambda_n)=\tr(Z_nF_n)$ and, for $0\le t\le n-1$,
\begin{align*}
	l_t(Z_t, \Sigma_t, \Lambda_t) =& \tr( (F+K_t^\top G K_t)Z_t) + \tr(Q^\top G_1 Q S_t) + \tr(D_t^\top G D_t \Sigma_t) \\
	& + 2\tr(D_t^\top GK_t L_t \Sigma_t) - 2\tr((D_t + K_tL_t)^\top G \tilde{I}QS_t^{\frac{1}{2}} \Sigma_{t}^{\frac{1}{2}}).
\end{align*}
We thus can formulate the optimization of communication power as an MDP, denoted by $\mathcal{M}$, with state $ (Z_t,\Sigma_t)$, action $\Lambda_t$, state transition equations \eqref{eq: cov-zt-2}-\eqref{eq: cov-sigmat-2}, and cost function $l_t(Z_t,\Sigma_t,\Lambda_t)$ at time $t$. The objective is to find a policy to minimize the total cost over the time horizon:
\begin{align}
	\min_{\{\Lambda_t \}} \ \sum_{t=0}^n 	l_t(Z_t, \Sigma_t, \Lambda_t). 
\end{align} 

Since $f^Z_t,f^\Sigma_t$, and $l_t$ are all nonlinear functions, it is generally difficult to derive a closed-form expression for the optimal solution to the MDP. However, by imposing a mild boundary condition on the estimation error covariance, we can derive the corresponding optimal solution using the discrete matrix minimum principle. In particular, we consider the following boundary condition:
\begin{align}  \label{eq: bd-con}
	\Sigma_n = \epsilon \Sigma_0, 
\end{align}
where $\epsilon>0$ is  a constant. Intuitively, the overall control cost cannot be minimized if the follower maintains a large estimation error of the target state throughout the time horizon. Therefore, any optimal implicit communication scheme must ensure that the estimation error covariance at the final time step is small. To reflect this, we approximate the desired terminal covariance by scaling the initial covariance, i.e., $\Sigma^*_n \approx \epsilon \Sigma_0$, with $\epsilon$ capturing the desired accuracy level. Consequently, the optimal solution to the MDP with the boundary condition \eqref{eq: bd-con} is near-optimal for the original problem.

To derive the optimal solution under the boundary condition, let us define the Hamiltonian function \cite{athans1966discrete}
\begin{align*}
	\mathcal{H}_t(Z_t,\Sigma_t,\Lambda_t,\theta_{Z,t+1},\theta_{\Sigma,t+1}) =& l_t(Z_t, \Sigma_t,\Lambda_t) + \tr(f^Z_t(Z_t,\Sigma_t,\Lambda_t)\theta_{Z,t+1}^\top) +  \tr(f^\Sigma_t(\Sigma_t,\Lambda_t) \theta^\top_{\Sigma,t+1}) \\
	=& \tr( (F+K_t^\top G K_t)Z_t) + \tr(\bar{A}_t Z_t\bar{A}_t ^\top \theta_{Z,t+1}) + \tr(W\theta_{Z,t+1}) \\
	& + 2\tr(\Lambda^{\frac{1}{2}}_t U^\top \Sigma_{t}^{\frac{1}{2}} L_{t+1}^\top \theta_{Z,t+1}Q_1 U)  - 2\tr(\Lambda^{\frac{1}{2}}_t U^\top \Sigma_{t}^{\frac{1}{2}} (D_t + K_tL_t)^\top G \tilde{I}Q U)  \\
	&+ \tr(\Sigma_{t}^{\frac{1}{2}} U(I+\Lambda_tH)^{-1}U^\top \Sigma_{t}^{\frac{1}{2}} \theta_{\Sigma,t+1})  \\
	& + \tr(\Lambda_t U^\top Q_1^\top (\theta_{Z,t+1}+G_1)Q_1U) - \tr(BD_t \Sigma_t(L_{t+1}^\top + L^\top_t \bar{A}^\top_t)\theta_{Z,t+1}) \\
	& + \tr(D_t^\top G D_t \Sigma_t) + 2\tr(D_t^\top GK_t L_t \Sigma_t),
\end{align*}
where $\theta_{Z,t}\in \mathbb{R}^{d_0\times d_0}$ is the costate associated with $Z_t$, and $\theta_{\Sigma,t}\in \mathbb{R}^{d_0\times d_0}$ is the costate associated with $\Sigma_t$.
The following conditions are necessary for the optimal solution:
\begin{align} \label{eq:theta-Z}
	\theta^*_{Z,t} =& \frac{\partial \mathcal{H}_t}{\partial Z_t} = F+K_t^\top G K_t +  \bar{A}_t^\top \theta^*_{Z,t+1} \bar{A}_t, \ \theta^*_{Z,n} = F_n , \\
	\theta^*_{\Sigma,t} =& \frac{\partial \mathcal{H}_t}{\partial \Sigma_t} = D_t^\top GD_t + D_t^\top GK_t L_t + L^\top_t K_t^\top G D_t + L^\top_{t+1} \theta^*_{Z,t+1} BD_t + D_t^\top B^\top \theta^*_{Z,t+1} \bar{A}_t L_t \notag \\ \label{eq:theta-sigma}
	&+ \Theta_1 + 2\Theta_2 - 2 \Theta_3, \quad \theta^*_{\Sigma,n} = 0, \\ \label{eq:cond-sigma}
	\Sigma^*_{t+1}  =& \Sigma_{t}^{*\frac{1}{2}} U(I+\Lambda^*_tH)^{-1}U^\top   \Sigma_{t}^{*\frac{1}{2}}, \\ \label{eq:gradient-lam}
	0 = & \frac{\partial \mathcal{H}_t}{\partial \Lambda_t}|_{\Lambda^*_t},
\end{align}
where $\Theta_1$, $\Theta_2$, and $\Theta_3$ are determined by the following Sylvester-Lyapunov equations, respectively:
\begin{align} \label{eq:SL1}
	&\Sigma_{t}^{*\frac{1}{2}}\Theta_1 + \Theta_1 \Sigma_{t}^{*\frac{1}{2}} = \theta^*_{\Sigma,t+1} \Sigma_{t}^{*\frac{1}{2}} V_t +V_t \Sigma_{t}^{*\frac{1}{2}} \theta^*_{\Sigma,t+1}, \\ \label{eq:SL2}
	&\Sigma_{t}^{*\frac{1}{2}}\Theta_2 +\Theta_2 \Sigma_{t}^{*\frac{1}{2}} = \frac{1}{2}(L_{t+1}^\top \theta^*_{Z,t+1}Q_1 S^{\frac{1}{2}}_t +   S^{\frac{1}{2}}_t Q_1^\top\theta^*_{Z,t+1}L_{t+1} ), \\  \label{eq:SL3}
	&\Sigma_{t}^{*\frac{1}{2}}\Theta_3 + \Theta_3 \Sigma_{t}^{*\frac{1}{2}} = \frac{1}{2}((D_t + K_tL_t)^\top G \tilde{I}Q S^{\frac{1}{2}}_t +   S^{\frac{1}{2}}_t (G \tilde{I}Q)^\top (D_t + K_tL_t)).
\end{align}
The derivation of \eqref{eq:theta-sigma} is detailed in Appendix.
Since $\Lambda_t$ is a diagonal matrix, it is enough to compute the derivative w.r.t each diagonal entries. Therefore, \eqref{eq:gradient-lam} can be expressed as follows:for $j=1,2,\cdots,d_0$,
\begin{align} \label{eq:gradient-lam-j}
	0=\frac{\partial \mathcal{H}_t}{\partial \Lambda_t(j)} |_{\Lambda^*_t}= \frac{M_{1,t}(j)}{\sqrt{\Lambda^*_t(j)}}  - \frac{H(j) M_{2,t}(j)}{[1+\Lambda^*_t(j)H(j)]^2} 
	+M_{3,t}(j), 
\end{align}
where $M_{k,t}(j)$ denotes the $j$-th diagonal entry of the matrix $M_{k,t}$ for $k=1,2,3$, with the coefficient matrices given by
\begin{align*}
	M_{1,t} \triangleq & U^\top \Sigma_{t}^{*\frac{1}{2}} [L_{t+1}^\top \theta^*_{Z,t+1}Q_1+ (D_t + K_tL_t)^\top G \tilde{I}Q] U, \\
	M_{2,t} \triangleq &  U^\top \Sigma_{t}^{*\frac{1}{2}} \theta^*_{\Sigma,t+1} \Sigma_{t}^{*\frac{1}{2}} U, \\
	M_{3,t} \triangleq &  U^\top Q_1^\top (\theta^*_{Z,t+1}+G_1)Q_1U.
\end{align*}
The optimal costate sequence $\{\theta^*_{Z,t}\}$ can be computed recursively using \eqref{eq:theta-Z}.
Given the boundary condition \eqref{eq: bd-con}, we first solve for $\Lambda^*_{n-1}$ and $\Sigma^*_{n-1}$ using equations \eqref{eq:cond-sigma} and \eqref{eq:gradient-lam-j}. With these values, we then solve the Sylvester-Lyapunov equations \eqref{eq:SL1}-\eqref{eq:SL3} to compute $\theta^*_{\Sigma,n-1}$. This procedure is repeated backward in time from $t=n-1$ to $t=0$, allowing us to obtain all optimal communication power diagonal matrices $\{\Lambda^*_t\}$.

\subsection{A Computationally Efficient Solution}
Although incorporating the boundary condition \eqref{eq: bd-con} enables us to derive a near-optimal solution to the communication power optimization problem, solving equations \eqref{eq:cond-sigma}, \eqref{eq:gradient-lam-j}, and the Sylvester-Lyapunov equations \eqref{eq:SL1}–\eqref{eq:SL3} remains computationally intensive and analytically tedious. To address this, we propose a more efficient approach by further restricting the structure of the matrices $\{\Lambda_t\}$. In particular,
we consider the MDP input of the form
\begin{align} \label{eq:lam-a}
	\Lambda_t = {a}_t H^{-1}, \ 0\le t \le n-1,
\end{align}
where ${a}_t>0$ is a real value. Then $V_t$ admits a simple form:
\begin{align*}
	V_t = U(I+\Lambda_t H)^{-1} U^\top = \frac{1}{1+{a}_t} I.
\end{align*}
Let $\bar{a}_t \triangleq 1/(1+{a}_t)$ and $b_{t+1} = \prod_{i=0}^{t}\bar{a}_t $ with $b_0=1$. Then $\bar{a}_t\in (0,1)$ and
\begin{align} \label{eq: sigmat}
	\Sigma_{t+1}  = \bar{a}_t\Sigma_{t} = b_{t+1} \Sigma_0.
\end{align}
With this restriction on the communication power,
the MDP defined in \eqref{eq: cov-zt-2}-\eqref{eq: cov-sigmat-2} can be further simplified by representing it using state $Y_t=(Z_t, b_t)$ and input $a_t$, as follows:
\begin{align}
	b_{t+1} & = f^b_t(b_t,a_t)\triangleq b_t/(1+a_t), \\
	Z_{t+1} & = f^Z_t(Z_t,b_t,a_t) \triangleq  \bar{A}_t Z_t\bar{A}_t ^\top + \sqrt{{a}_t b_t} Q_{ab,t} - b_t Q_{b,t}+ {a}_t Q_a +W,
\end{align}
where the constants are given by
\begin{align*}
	Q_a & \triangleq Q_1 UH^{-1}U^\top Q_1^\top, \\
	Q_{b,t} & \triangleq BD_t \Sigma_0 L_{t+1} + \bar{A}_tL_t\Sigma_0(BD_t )^\top, \\
	Q_{ab,t} & \triangleq Q_1UH^{-\frac{1}{2}}U^\top\Sigma_0^{\frac{1}{2}}L^\top_{t+1} + (Q_1UH^{-\frac{1}{2}}U^\top\Sigma_0^{\frac{1}{2}} L^\top_{t+1})^\top.
\end{align*}
Accordingly, the cost function is then given by
\begin{align} \label{eq: dp-cost}
	l_t(Z_t, b_t, a_t) = \tr( (F+K_t^\top G K_t)Z_t) + a_t r_a + b_t r_{b,t} - 2 \sqrt{a_t b_t} r_{ab,t},
\end{align}
where the constants are defined as
\begin{align*}
	r_a & \triangleq \tr(Q^\top G_1 Q UH^{-1}U^\top), \\
	r_{b,t} & \triangleq \tr(D_t^\top G (D_t \Sigma_0+2K_tL_t\Sigma_0)), \\
	r_{ab,t} & \triangleq\tr((D_t + L_tK_t)^\top G \tilde{I}QUH^{-\frac{1}{2}}U^\top\Sigma_0^{\frac{1}{2}}).
\end{align*}

Using the same method as before, we define the Hamiltonian function as
\begin{align*}
	\mathcal{H}_t(Z_t,b_t,a_t,\theta_{Z,t+1},\theta_{b,t+1}) =& l_t(Z_t, b_t, a_t) + \tr(f^Z_t(Z_t,b_t,a_t)\theta_{Z,t+1}^\top) + \theta_{b,t+1} f^b_t(b_t,a_t) \\
	=& \tr( (F+K_t^\top G K_t)Z_t) + \tr(\bar{A}_t Z_t\bar{A}_t ^\top \theta_{Z,t+1}^\top) + \tr(W\theta_{Z,t+1}^\top) \\
	 & + a_t  [r_a  +  \tr(Q_a \theta_{Z,t+1}^\top) ] + \sqrt{a_tb_t} [\tr(Q_{ab,t} \theta_{Z,t+1}^\top) -2 r_{ab,t} ]\\
	 & + b_t [r_{b,t} + \theta_{b,t+1}/(1+a_t) - \tr(Q_{b,t} \theta_{Z,t+1}^\top) ],
\end{align*}
where $\theta_{Z,t}\in \mathbb{R}^{d_0\times d_0}$ is the costate associated with $Z_t$, and $\theta_{b,t}\in \mathbb{R}$ is the costate associated with $b_t$. The necessary conditions of the optimal solution are, respectively,
\begin{align} \label{eq:theta-Z2}
	\theta^*_{Z,t} &= \frac{\partial \mathcal{H}_t}{\partial Z_t} = F+K_t^\top G K_t +  \bar{A}_t^\top \theta^*_{Z,t+1} \bar{A}_t, \ \theta^*_{Z,n} = F_n  \\ \label{eq:theta-b}
	\theta^*_{b,t} &= \frac{\partial \mathcal{H}_t}{\partial b_t} = r_{b,t} + \theta^*_{b,t+1}/(1+a_t) - \tr(\theta^*_{Z,t+1} Q^\top_{b,t} ) + \frac{a_t}{2 \sqrt{a_tb_t}} [\tr(\theta^*_{Z,t+1} Q_{ab,t} ^\top) -2 r_{ab,t} ],\ \theta^*_{b,n} = 0 \\ \label{eq:opt-b}
	b^*_{t+1} &= b^*_t/(1+a^*_t) \\  \label{eq:grad-a}
	0 &= \frac{\partial \mathcal{H}_t}{\partial a_t} 
	= r_a  +  \tr( \theta^*_{Z,t+1} Q_a^\top) +  \frac{b^*_t}{2 \sqrt{a^*_tb^*_t}} [\tr(\theta^*_{Z,t+1} Q_{ab,t} ^\top) -2 r_{ab,t} ] - b^*_{t+1}\frac{ \theta^*_{b,t+1}}{(1+a^*_t)}.
\end{align}
Given the input of the form in \eqref{eq:lam-a}, the boundary condition \eqref{eq: bd-con} reduces to $b^*_n=\epsilon$.
We then can compute the optimal sequence $\{a^*_t\}$ in backward recursion from $t=n-1$ to $t=0$ by solving the following equation:
\begin{align*}
	r_a  +  \tr( \theta^*_{Z,t+1} Q_a^\top) +  \frac{\sqrt{b^*_{t+1}(1+a_t^*)}}{2 \sqrt{a^*_t}} [\tr(\theta^*_{Z,t+1} Q_{ab,t} ^\top) -2 r_{ab,t} ] - b^*_{t+1}\frac{ \theta^*_{b,t+1}}{1+a_t^*} = 0.
\end{align*}
Once $a^*_t$ is determined, we compute $b^*_t$ and $\theta^*_{b,t}$ using \eqref{eq:opt-b} and \eqref{eq:grad-a}, respectively. This allows us to proceed to the next step and compute $a^*_{t-1}$.

\section{Under-actuated Leader} \label{sec:under}
In this section, we extend the implicit coordination scheme to systems with an under-actuated leader—that is, when $\rk(B_1) = r < d_0$. As shown in Theorem \ref{thm:covet}, in this case, the trace of the estimation error covariance, $\tr(\Sigma_t)$, is no longer guaranteed to be strictly decreasing over time when using Algorithm 1. This means that the follower's estimate of the target state can not be made arbitrarily accurate. To address this, we propose a modification to the implicit coordination scheme that restores the property of strictly decreasing estimation error.
 
First of all, consider the singular value decomposition (SVD) of matrix $B_1$:
\begin{align*}
	B_1 = \Gamma_0 \bar{\Psi}_1 \Gamma^\top_1,  
\end{align*}
where ${\Gamma}_0\in \mathbb{R}^{d_0\times d_0}$  and $\Gamma_1\in \mathbb{R}^{d_1\times d_1}$ are unitary matrices, and $\bar{\Psi}_1$ is a rectangular diagonal matrix, which can be expressed as a block matrix:
\begin{align*}
	\bar{\Psi}_1 = \begin{bmatrix}
		{\Psi}_1 & 0 \\
		0 & 0
	\end{bmatrix},
\end{align*}
where ${\Psi}_1 \in \mathbb{R}^{r\times r}$ is a diagonal matrix, whose diagonal entries are the positive singular values of $B_1$. Now, the implicit communication channel, as defined in \eqref{eq:output-yt}, is equivalent to the following:
\begin{align} \label{eq:ybar-t}
	\bar{y}_t \triangleq {\Gamma}^\top_0 y_t = \bar{\Psi}_1 \Gamma^\top_1 s_t + {\Gamma}^\top_0 w_t = \bar{\Psi}_1 \bar{s}_t + \bar{w}_t,
\end{align}
where $ \bar{s}_t = \Gamma^\top_1 s_t\in \mathbb{R}^{d_1}$ and $\bar{w}_t ={\Gamma}^\top_0 w_t \in \mathbb{R}^{d_0}$ with covariance matrix $\cov(\bar{w}_t)=  \bar{W}\triangleq  {\Gamma}^\top_0 W {\Gamma}_0$. 

It is clear that only the first $r$ elements of $\bar{s}_t$ are useful for communication, while the remaining elements effectively transmit no information to $\bar{y}_t$.
For simplicity, we assume that $\tau\triangleq d_0/r$ is an integer\footnote{If $d_0/r$ is not an integer, we can construct an augmented message vector $\bar{x}_*$ by padding zeros after $x_*$, i.e., $\bar{x}_* = [x_*, 0]$.} and define a set of $r\times d_0$ projection matrices
\begin{align} \label{eq:proj-mat}
	P_k = \begin{bmatrix}
		 0&I_r& 0
	\end{bmatrix}, \ k\in \{0,1,\cdots,\tau-1\},
\end{align}
where $I_r$ denotes the $r\times r$ identity matrix, positioned from the $(kr+1)$-th to the $(k+1)r$-th column. 
Furthermore, we express $\bar{W}$ as a block matrix:
\begin{align*}
	\bar{W} = \begin{bmatrix}
		\bar{W}_1 & \bar{W}_3 \\
		\bar{W}_3^\top & \bar{W}_2
	\end{bmatrix},
\end{align*}
where $\bar{W}_1$ is a $r\times r$ covariance matrix corresponding to the first $r$ elements of $\bar{w}_t$.

\subsection{Implicit Coordination Scheme for the Under-actuated Setting}

In this section, we first define a virtual channel and propose a corresponding coding scheme. Building on this construction, we then introduce the implicit coordination scheme for the under-actuated setting and show that it guarantees a strictly decreasing estimation error covariance for the follower.

Specifically, we construct a virtual channel with an $r$-dimensional input and output, defined as:
\begin{align} \label{eq:vc-k}
	\tilde{y}_t = \Psi_1 \tilde{s}_t + \tilde{w}_t,
\end{align}
where  $\tilde{w}_t\in \mathbb{R}^{r}$ is a zero-mean white Gaussian noise with covariance matrix $\bar{W}_1 $. This virtual channel is derived directly from the one defined in \eqref{eq:ybar-t} by extracting only the first $r$ entries from its input and output—since the remaining entries carry no information, they are discarded. We focus on this virtual channel rather than the original in \eqref{eq:ybar-t} because it is more tractable for analysis.

Assume now that we want to transmit a message $e_t \in \mathbb{R}^{d_0}$ over the virtual channel at time $t$, where $e_t$ follows a Gaussian distribution $\mathcal{N}(0,\Sigma_t)$. Since the dimension of the message exceeds that of the channel input, we apply a projection matrix $P_k$ as defined in \eqref{eq:proj-mat}. As will become clearer later,  changing the value of $k$ over time is necessary to drive the estimation error to zero. For now, we fix an arbitrary $k\in \{0,1,\cdots,\tau-1\}$.
We then define the coding and decoding scheme for time $t$ by extending the approach from Section~\ref{subsec:coding}, as described below:

\begin{itemize}
	\item [1.] \textbf{Encoding}: The encoder determines the channel input based on $e_t$  using the following mapping:
	\begin{align} \label{eq:enc}
		\tilde{s}^{(k)}_t = S_t^{\frac{1}{2}} P_k  \Sigma_t^{-\frac{1}{2}} e_t,
	\end{align}
	where  $S_t \in \mathbb{R}^{r\times r}$ is the covariance matrix of $\tilde{s}^{(k)}_t$. Here, the superscript ${(k)}$ emphasizes the dependence on the projection matrix $P_k$.
	We assume that $S_{t}$ and $\Psi_1 \bar{W}^{-1}_1 \Psi_1$ can simultaneously be diagonalizable by a unitary matrix $U_1$. That is, $\Psi_1 \bar{W}^{-1}_1 \Psi_1 = U_1 \Pi_1 U_1^\top$ and $S_{t} = U_1 \Lambda_{t} U^\top_1$, where both $\Lambda_{t}$ and $\Pi_1$ are $r\times r$ diagonal matrices.
	
	\item [2.] \textbf{Decoding}: The receiver observes the channel output
	\begin{align*}
		\tilde{y}_t^{(k)} = \Psi_1 S_t^{\frac{1}{2}} P_k  \Sigma_t^{-\frac{1}{2}} e_t + \tilde{w}_t,
	\end{align*}
	 and then estimates $e_t$ as follows:
	\begin{align} \label{eq:dec-vc}
		\hat{e}_t & \triangleq \mathbb{E}[e_t|\tilde{y}_t^{(k)}] = \cov(e_t) (\Psi_1 S_t^{\frac{1}{2}} P_k  \Sigma_t^{-\frac{1}{2}} )^\top \cov(\tilde{y}_t)^{-1} \tilde{y}_t \notag\\
		&=\Sigma_t^{\frac{1}{2}}  P^\top_k S_t^{\frac{1}{2}} \Psi_1(\Psi_1 S_t \Psi_1 + \bar{W}_1)^{-1}\tilde{y}_t^{(k)}.
	\end{align}
	Denote by $\varepsilon_{t} = e_{t} - \hat{e}_{t}$ the receiver's estimation error.  
\end{itemize}

Let $U_\tau \triangleq \mathtt{diag}(U_1, \cdots, U_1)$ be a block diagonal matrix with $\tau$ copies of $U_1$ on the diagonal.  The following lemma characterizes the distribution of the estimation error $\varepsilon_t$.

\begin{lemma} \label{lem:virtual-est}
	Given $\Psi_1 \bar{W}^{-1}_1 \Psi_1 = U_1 \Pi_1 U_1^\top$ and $S_{t} = U_1 \Lambda_{t} U^\top_1$. If $P_k$ is used as the projection matrix in \eqref{eq:enc}, then  the estimation error $\varepsilon_t$ is given by
	\begin{align*}
		\varepsilon_t =\Sigma_{t}^{\frac{1}{2}} \bar{V}^{(k)}_t  \Sigma_{t}^{-\frac{1}{2}} e_{t} -  \Sigma_{t}^{\frac{1}{2}}  \bar{V}^{(k)}_t  P^\top_k S_t^{\frac{1}{2}} \Psi_1 \bar{W}^{-1}_1 \tilde{w}_t,
	\end{align*}
	where $\bar{V}^{(k)}_t = U_\tau \tilde{V}^{(k)}_t U_\tau^\top$ and $\tilde{V}^{(k)}_t = \mathtt{diag}(I_{(k-1)r}, \tilde{V}_{t},I_{(\tau-k)r})$ is a block diagonal matrix. Each $\tilde{V}_{t} \in \mathbb{R}^{r\times r}$ is a diagonal matrix with the $j$-th diagonal element given by
	\begin{align*}
		\tilde{V}_{t}(j) = \frac{1}{1+ \Lambda_{t}(j) \Pi_1(j)}.
	\end{align*}
	In addition, $\varepsilon_{t}\sim \mathcal{N}(0, \Sigma_{t}^{\frac{1}{2}} \bar{V}^{(k)}_t \Sigma_{t}^{\frac{1}{2}})$.
\end{lemma}
\begin{IEEEproof}
	See Appendix.
\end{IEEEproof}

Lemma \ref{lem:virtual-est} characterizes how a transmission over the virtual channel improves the receiver's estimate of the message $e_t$. To see this, we examine the quantity
\begin{align*}
	\Sigma_t  - \cov(\varepsilon_t) = \Sigma_{t}^{\frac{1}{2}} U_\tau (I-\tilde{V}^{(k)}_t) U_\tau^\top \Sigma_{t}^{\frac{1}{2}} \succeq 0.
\end{align*}
It is clear that the estimation error covariance $\cov(\varepsilon_t)$ depends on the projection matrix $P_k$. Intuitively, because $P_k$ selects only part of the message $e_t$, the virtual channel input $s_t^{(k)}$ conveys information about only a subset of its components. Consequently, if the same projection matrix is used throughout the time horizon, only that subset can be driven toward zero estimation error, while the remaining components stay bounded away from zero. To overcome this limitation, the projection matrix must vary over time so that every dimension of the target state is updated sufficiently often.
In particular, we divide the time horizon into $p=\lceil n/\tau \rceil$ periods, each consisting of $\tau$ steps. Within each period, the projection matrices $\{P_k:0\le k\le \tau-1\}$ are applied sequentially, so that different components of the message are transmitted in different steps.

Using the same notations as in the full-actuated setting, we denote by $e_t\in \mathbb{R}^{d_0}$ the message to be transmitted at the $t$-th time step, where $e_0=x_*$ and $e_{t+1}$ will be computed recursively from $e_t$. In addition, denote by $\Sigma_t$ the covariance matrix of $e_t$. We now describe how the control inputs are determined at each time step.

As we discussed around \eqref{eq:ybar-t}, when $\rk(B_1)=r$, only the first $r$ elements of $\bar{s}_t$ can effectively transmit information, since $\bar{\Psi}_1$ has non-zero entries only in its upper-left $r\times r$ block. The remaining $d_1 - r$ elements of $\bar{s}_t$ thus carry no information. As a result, each time we use the implicit channel defined in  \eqref{eq:ybar-t}, it suffices to encode information only in the first $r$ elements and set the remaining $d_1-r$ elements to zero. Doing so effectively yields a virtual channel, as defined in \eqref{eq:vc-k}. 
Therefore, at time step $t$, the input to the implicit channel defined in  \eqref{eq:ybar-t} is constructed as
\begin{align*}
	\bar{s}_{t}=[\tilde{s}_t, \ \mathbf{0}] = [S_t^{\frac{1}{2}} P_k  \Sigma_t^{-\frac{1}{2}} e_t, \ \mathbf{0}],
\end{align*}
where $\tilde{s}_t$ corresponds to the input to the virtual channel and the remaining $d_1-r$ elements are padded with zeros. The index $k$ of the projection matrix depends on the time step $t$ and cycles through the set $\{0,1,\cdots,\tau-1\}$. Here, we simply set $k=t\ \mathtt{mod}\ \tau$.

Based on $\bar{s}_t$, the input to the original implicit channel \eqref{eq:output-yt} is given by $s_{t}=\Gamma_1 \bar{s}_{t}$. Following the same approach as in the fully-actuated setting, the control inputs at time $t$ are defined as
\begin{align}
	v_t = -K_t^l x_t + D_t^l x^{(t)}_* + s_t, \ q_t = -K_t^fx_t + D_t^f x_*^{(t)}. 
\end{align}
Upon observing $x_{t+1}$, we first compute $y_{t}$ as follows
\begin{align}
	y_t = x_{t+1} - (A-BK_t)x_t - BD_t x_*^{(t)}.
\end{align}
The output of the equivalent channel \eqref{eq:ybar-t} is then given by $\bar{y}_{t}=\Gamma_0^\top y_{t}$. Since only the first $r$ element carry information about the message, we extract the informative sub-vector $\tilde{y}_{t}= \bar{y}_{t}[1:r]$ and discard the remaining elements.

The next step is to estimate $e_t$ from $\tilde{y}_t$ and the optimal estimate $\hat{e}_{t}=\mathbb{E}[e_t|\tilde{y}_t]$ is given by \eqref{eq:dec-vc}. Denote by the estimation error $e_{t+1}=e_t - \hat{e}_t$ and its covariance $\Sigma_{t+1}=\cov(e_{t+1})$; they are characterized in Lemma \ref{lem:virtual-est}. Finally, as in the fully-actuated setting, we can update the follower's estimate of the target state estimate via $x_*^{(t+1)} = x_*^{(t)}  + \hat{e}_{t} $  and the offset for both agents via $c_{t+1} = D_{t+1}x_*^{(t+1)}$. The new estimation error $e_{t+1}$ is transmitted at the next time step, and the procedure is repeated until the end of the time horizon. The full procedure is summarized in Algorithm~\ref{alg:2}.

\begin{algorithm}[!t]
	\caption{Generalized Implicit Coordination Scheme (under-actuated leader)}
	\label{alg:2}
	\SetAlgoLined
	\KwIn{Target state $x_*$, horizon $n$, integer $\tau = \lceil d_0/r \rceil$}
	
	\textbf{Initialization:} $ c_0 = 0, e_{0} = x_*, k=0$
	
	\For{$t = 0$ \KwTo $n-1$}{
		$k= t \ \mathtt{mod} \ \tau$ \\
		Compute the input to the virtual channel:\
	\[
	\tilde{s}_t = S^{\frac{1}{2}}_tP_k \Sigma_{t}^{-\frac{1}{2}}e_{t}
	\] \\
		 Compute the communication signal:\
		\[ 
			s_{t} =\Gamma_1 [\tilde{s}_t, \ \mathbf{0}] \qquad \text{// $d_1-r$ zeros are padded}
		\] \\
		Compute the control inputs:\
		\[ 
		v_t =-K^l_t x_t  + c^l_t + s_{t}, \ q_t = -K_t^f x_t + c^f_t
		\] \\
		The leader inputs $v_t$  and the follower inputs $q_t$; both agents observe the new state $x_{t+1}$ \\
		Compute the channel output of the virtual channel:\
		$$y_{t}=\Gamma_0^\top(x_{t+1}- (A-BK_t)x_t - Bc_t),\ \tilde{y}_{t}= y_{t}[1:r]$$ \\
		Estimate $\hat{e}_{t}=\mathbb{E}[e_t|\tilde{y}_t]$ as follows:
		$$\hat{e}_{t}= \Sigma_t^{\frac{1}{2}}  P^\top_k S_t^{\frac{1}{2}} \Psi_1(\Psi_1 S_t \Psi_1 + \bar{W}_1)^{-1}\tilde{y}_t $$ \\		
		
		Update the estimation error and the corresponding error covariance:\
		$$e_{t+1} = e_t - \hat{e}_{t}, \  \Sigma_{t+1} = \Sigma_{t}^{\frac{1}{2}} \bar{V}^{(k)}_t \Sigma_{t}^{\frac{1}{2}}$$ \\
		
		Update the target state estimate $x_*^{(t+1)} = x_*^{(t)}  + \hat{e}_{t} $ and offset:
		$$c_{t+1}=[c^l_{t+1}, c^f_{t+1}] = D_{t+1}x_*^{(t+1)}$$ \\
	}

\end{algorithm}

Although the follower’s estimation error under Algorithm~\ref{alg:2} is not guaranteed to decrease in every dimension at each time step, it decreases strictly across periods. Consequently, given a sufficiently long time horizon, the follower can obtain an arbitrarily accurate estimate of the target state. This result is formalized in the following theorem.

\begin{theorem} \label{thm: error-ua}
	Let  $\Sigma_t$ denote the follower's estimation error covariance of $x_*$ at time $t$ under Algorithm 2. If there is a $\sigma>0$ such that $\Lambda_t(j)\ge \sigma$ for all $t$ and $j$, then for any $i\ge 0$,  
	\begin{align*}
		\tr(\Sigma_{(i+1)\tau}) \le \frac{1+\sigma \pi }{1+2\sigma \pi} \tr(\Sigma_{i\tau}),
	\end{align*}
	where $\pi = \min \Pi_1(j)$ is the minimum diagonal entry of $\Pi_1$.
\end{theorem}
\begin{IEEEproof}
	Note that
	\begin{align*}
		\Sigma_{i\tau} - \Sigma_{(i+1)\tau} & = \sum_{k=0}^{\tau-1} (\Sigma_{i\tau + k} - \Sigma_{i\tau + k + 1} ) \\
		& \overset{(a)}{=} \sum_{k=0}^{\tau-1} \left[ \Sigma^{\frac{1}{2}}_{i\tau + k} (I - \bar{V}_{i\tau + k}^{(k)}) \Sigma^{\frac{1}{2}}_{i\tau + k} \right],
	\end{align*}
	where (a) follows from Lemma~\ref{lem:virtual-est}. According to the definition of $\bar{V}_{i\tau + k}^{(k)}$ in Lemma~\ref{lem:virtual-est}, it is easy to see that $\Sigma_{i\tau + k} \succeq \Sigma_{i\tau + k + 1} $ for all $k$. As a result, $\Sigma_{i\tau + k} \succeq \Sigma_{(i+1)\tau  } $ for all $k$. We then have
	\begin{align} \label{eq: tr-diff}
		\tr(\Sigma_{i\tau}) - \tr(\Sigma_{(i+1)\tau}) & = \sum_{k=0}^{\tau-1} \tr\left(\Sigma^{\frac{1}{2}}_{i\tau + k} (I - \bar{V}_{i\tau + k}^{(k)}) \Sigma^{\frac{1}{2}}_{i\tau + k}  \right) \notag \\
		& = \sum_{k=0}^{\tau-1} \tr\left( (I - \bar{V}_{i\tau + k}^{(k)}) \Sigma_{i\tau + k}  \right) \notag \\
		& \ge \sum_{k=0}^{\tau-1} \tr\left( (I - \bar{V}_{i\tau + k}^{(k)}) \Sigma_{(i+1)\tau }  \right) \notag \\
		& = \tr\left(\sum_{k=0}^{\tau-1}(I - \bar{V}_{i\tau + k}^{(k)}) \Sigma_{(i+1)\tau }  \right),
	\end{align}
	where the inequality follows from the fact that $I - \bar{V}_{i\tau + k}^{(k)} \succeq 0$ and $\Sigma_{(i\tau + k)} \succeq \Sigma_{(i+1)\tau } $. To see this, note that, for any PSD matrices $M\succeq M' \succeq 0$,
	\begin{align*}
		&\tr\left( (I - \bar{V}_{i\tau + k}^{(k)}) M \right) - \tr\left( (I - \bar{V}_{i\tau + k}^{(k)}) M' \right) \\
		=&\tr\left( (I - \bar{V}_{i\tau + k}^{(k)}) (M-M') \right)  \\
		=& \tr\left( (M-M')^{\frac{1}{2}}(I - \bar{V}_{i\tau + k}^{(k)}) (M-M')^{\frac{1}{2}} \right) \ge 0.
	\end{align*}
	Now, according to the definition of $\bar{V}_{i\tau + k}^{(k)}$,
	\begin{align*}
		\sum_{k=0}^{\tau-1}\bar{V}_{i\tau + k}^{(k)} = U_\tau \sum_{k=0}^{\tau-1}\tilde{V}_{i\tau + k}^{(k)} U_\tau^\top = U_\tau M_i U_\tau^\top,
	\end{align*}
	where $M_i\in \mathbb{R}^{d_0\times d_0}$ is a diagonal matrix, with the $(kr+j)$-th diagonal entry given by
	\begin{align*}
		M_i(kr+j) = \frac{1}{1+S_{i\tau + k}(j) \Pi_1(j)}, \quad 0\le k \le \tau-1, 1\le j \le r.
	\end{align*}
	Let $\tilde{\Sigma}_{(i+1)\tau} = U_\tau^\top \Sigma_{(i+1)\tau} U_\tau$.
	It then follows from \eqref{eq: tr-diff} that
	\begin{align}
		\tr(\Sigma_{i\tau}) - \tr(\Sigma_{(i+1)\tau}) & \ge \tr((I-M_i) U_\tau^\top \Sigma_{(i+1)\tau} U_\tau) \\
		& = \sum_{j=1}^{d_0} [1-M_i(j)] \tilde{\Sigma}_{(i+1)\tau}(j) \\  \label{eq: tr-diff2}
		& \ge \left(1-\frac{1}{1+\sigma \pi}\right) \sum_{j=1}^{d_0}  \tilde{\Sigma}_{(i+1)\tau}(j) \\  \label{eq: tr-diff3}
		& = \frac{\sigma \pi }{1+\sigma \pi} \tr(\Sigma_{(i+1)\tau}), 
	\end{align}
	where the inequality~\eqref{eq: tr-diff2} follows from the assumption that $\Pi_1(j)\ge \pi$ and $S_{i\tau+k}(j)\ge \sigma$. Re-arranging the above inequality yields
	\begin{align*}
		\tr(\Sigma_{(i+1)\tau})  \le \frac{1+\sigma \pi }{1+2\sigma \pi} \tr(\Sigma_{i\tau}). 
	\end{align*}
	This completes the proof.
	
\end{IEEEproof}

\subsection{Communication Power Optimization}
This part focuses on the optimization of communication power, characterized by the sequence of covariance matrices $\{S_{t}\}$. As described in the previous subsection,  for analytical and computational tractability, we consider $S_{t}$ to be simultaneously diagonalizable with $\Psi_1 \bar{W}^{-1}\Psi_1$. As a result, we express $S_{t}=U_1 \Lambda_{t} U_1^\top$ for all $t$. Our goal is therefore to find the optimal sequence of diagonal matrices $\{\Lambda_{t} \}$ that minimizes the overall control cost $J_n$. 

Let $\bar{S}_t \triangleq \dig(S_t,0)\in \mathbb{R}^{d_1\times d_1}$. Although $\bar{S}_t$ does not admit a true matrix square root in the standard sense, for notational convenience we define $\bar{S}_t^{\frac{1}{2}} \triangleq \dig(S_t^{\frac{1}{2}},0)\in \mathbb{R}^{d_1\times d_1}$, so that $\bar{S}_t^{\frac{1}{2}}\bar{S}_t^{\frac{1}{2}}=\bar{S}_t$. In addition, let $\bar{P}_k = [P_k,0]\in \mathbb{R}^{d_1\times d_0}$. Then
\begin{align*}
	s_t = \Gamma_1 \bar{S}_t^{\frac{1}{2}} \bar{P}_k \Sigma_{t}^{-\frac{1}{2}}e_{t}.
\end{align*}
We thus can express the input at time $t$ under the implicit coordination scheme as follows:
\begin{align} \label{eq: ut-under}
	u_t &= -K_t x_t + D_t x_*^{(t)} + \tilde{I}{s}_t  \notag \\
	& = -K_t(x_t -x_*) + c'_t + (\tilde{I} \Gamma_1 \bar{S}_t^{\frac{1}{2}} \bar{P}_k \Sigma_{t}^{-\frac{1}{2}}  - D_t )e_t ,
\end{align}
where $c'_t$ and $\tilde{I}$ are defined under \eqref{eq: ut}. Using the same argument as in the fully actuated setting yields
\begin{align} \label{eq:error-zt-under}
	z_{t+1} = x_{t+1} - x_* =  \bar{A}_tz_t +  c''_t + (B_1\Gamma_1 \bar{S}_t^{\frac{1}{2}} \bar{P}_k \Sigma_{t}^{-\frac{1}{2}}  -BD_t )e_{t} + w_t,
\end{align}
where $\bar{A}_t = A-BK_t$ and $c''_t \triangleq (A-I)x_* + B c'_t $. We then can derive the quadratic cost for each $t<n$:
\begin{align} \label{eq:q-cost}
	\mathbb{E}[z_t^\top F z_t] + \mathbb{E}[{u}_t^\top G {u}_t]  =& \tr((F+K_t^\top GK_t)Z_t) + \tr(\Gamma_1^\top G_1 \Gamma_1 \bar{S}_t)+ \tr(D_t^\top G D_t \Sigma_t)  - 2\tr(D_t^\top G \tilde{I} \Gamma_1 \bar{S}_t^{\frac{1}{2}} \bar{P}_k \Sigma_{t}^{\frac{1}{2}} ) \notag\\
	& - 2\tr((\tilde{I} \Gamma_1 \bar{S}_t^{\frac{1}{2}} \bar{P}_k \Sigma_{t}^{-\frac{1}{2}} - D_t )^\top G K_t \Omega_t ) + \text{constant},
\end{align}
where the constant includes terms that are independent of the decision variable $\Lambda_t$.

Note that $\tilde{w}_t = P_1\Gamma^\top_0 w_t$. Using the same argument as in the derivation of \eqref{eq: cov-zt}-\eqref{eq: cov-pit} obtains
\begin{align} \label{eq: Zt-under}
	Z_{t+1}  = & \bar{A}_t Z_t\bar{A}_t ^\top + M + M^\top + B_1\Gamma_1\bar{S}_t\Gamma_1^\top B_1^\top  +BD_t \Sigma_t D_t^\top B^\top + W,   \\  
	\Sigma_{t+1}  =& \Sigma_{t}^{\frac{1}{2}} \bar{V}^{(k)}_t  \Sigma_{t}^{\frac{1}{2}},  \\   \label{eq: omegat-under}
	L_{t+1}  = &  \bar{A}_t L_t - B D_t + \Gamma_0 C_t^w  \Sigma_{t}^{-\frac{1}{2}} ,
\end{align}
where $L_t = \Omega_t \Sigma_t^{-1}$, $C^w_t\triangleq [0, \bar{W}_3^\top \bar{W}_1^{-1} \Psi_1 S^{\frac{1}{2}}_t P_k]$, and
\begin{align*}
	M\triangleq  B_1\Gamma_1 \bar{S}_t^{\frac{1}{2}} \bar{P}_k \Sigma_{t}^{\frac{1}{2}} (\bar{A}_t L_t - BD_t)^\top - BD_t \Sigma_{t}  L_t^\top \bar{A}_t^\top .
\end{align*}
By omitting the constant term in \eqref{eq:q-cost}, we define a cost function as follows: $l_t(Z_n,\Sigma_n, L_n, \Lambda_n) =\tr(Z_n F_n)$, and
\begin{align*}
	l_t(Z_t,\Sigma_t, L_t, \Lambda_t) =& \tr((F+K_t^\top GK_t)Z_t) + \tr(\Gamma_1^\top G_1 \Gamma_1 \bar{S}_t) + \tr(D_t^\top G D_t \Sigma_t)  \\
	& - 2\tr(D_t^\top G \tilde{I} \Gamma_1 \bar{S}_t^{\frac{1}{2}} \bar{P}_k \Sigma_{t}^{\frac{1}{2}} ) 
	 - 2\tr((\tilde{I} \Gamma_1 \bar{S}_t^{\frac{1}{2}} \bar{P}_k \Sigma_{t}^{-\frac{1}{2}} - D_t )^\top G K_t L_t \Sigma_t ), \ 0\le t\le n-1.
\end{align*}
The optimization of communication power then reduces to an optimal control problem (i.e., a deterministic MDP) with system dynamics given by \eqref{eq: Zt-under}-\eqref{eq: omegat-under}, cost function $l_t$, and the objective defined as
\begin{align*}
	\min_{\{\Lambda_t\}} \ \sum_{t=0}^{n} l_t(Z_t,\Sigma_t, L_t, \Lambda_t) .
\end{align*}
Unfortunately, due to the third term in \eqref{eq: omegat-under}, the above MDP can not be simplified without loss of optimality as in the fully actuated setting, making it difficult to derive the necessary conditions for the optimal solution. Nevertheless, the problem can still be tackled using differential dynamic programming (DDP) methods \cite{jacobson1970differential}, such as gradient-based iterative LQR (iLQR), or reinforcement learning methods such as soft actor-critic (SAC) \cite{SAC} and proximal policy optimization (PPO) \cite{PPO}. 

\begin{remark}
	In Algorithm~\ref{alg:2}, the projection matrices $P_k$ are applied in a round-robin manner, following the default order $k=0, 1, \cdots, \tau -1$. In practice, however, any ordering can be used, and Theorem~\ref{thm: error-ua} still holds as long as the round-robin scheme is maintained. It is worth noting that the choice of ordering may affect the overall control cost, and the default order considered here is not necessarily optimal.
\end{remark}

\section{Numerical Results} \label{sec:exp}
This section presents numerical results demonstrating the effectiveness of the proposed implicit coordination scheme. The results verify the theoretical insights developed in the previous sections. In the fully actuated setting, although the leader can drive the system to the target state without assistance from the follower, incorporating implicit communication to coordinate the two controllers significantly reduces the overall control cost.  In the under-actuated setting, the target state cannot be achieved without coordination; however, implicit communication enables the controllers to reach the target state with a small control cost.  The experiments are implemented in Python, and the results shown in the figures are averaged over 50 independent runs.

\subsection{Fully Actuator Setting}
We first conduct experiments on a fully actuated system with dynamics defined as
\begin{align*}
	A = \begin{bmatrix}
	1.5& 0.2& 0& 0.7 \\
	0& 0.5& 0.5& 0.3 \\
	0.2& 0& 1.9& 0.4 \\
	0.3& 0& 0.3& 1.7
	\end{bmatrix},
	B_1 = \begin{bmatrix}
		1& 2& 0& 0 \\
		0& 0& 1& 0 \\
		0& 1.2& 0& 0 \\
		0& 0& 0& 1.3
	\end{bmatrix},
	B_2 = \begin{bmatrix}
		0& 0 \\
		1& 0 \\
		2& 0 \\
		0& 3
	\end{bmatrix}.
\end{align*}
The noise covariance is $W=\dig(0.1, 0.1, 0.1, 0.1)$. Moreover, the matrices defining the quadratic cost are $F=\dig(2,1,1,2)$, $F_n = 10 I, G_1=\dig(2,2,4,6)$, and $G_2 = \dig(2,2)$. Finally, the distribution of the target state is defined as $\mathcal{N}(0,\Sigma_0)$, where $\Sigma_0 = \dig(5,5,5,5)$.

We compare our method with two benchmarks: the first is control with explicit communication (denoted by \texttt{Ex-Comm} in the figures), which assumes that the two controllers can communicate through an explicit communication channel, such that the follower can receive the target state perfectly from the leader at the beginning. Under this assumption, the problem reduces to a standard LQG control problem whose optimal control law can be computed easily. The resulting control cost of \texttt{Ex-Comm} therefore serves as a lower bound on the optimal control cost of our problem. The second benchmark is control solely by the leader (\texttt{Leader-only}), where the follower applies zero input and the leader applies the optimal control law for the system $(A,B_1)$, as if the follower were absent. \texttt{Leader-only} is a natural choice for control without communication.

For our method---control with implicit communication (\texttt{Im-Comm})---we consider two implementations that differ in how the communication power is determined. The first implementation, denoted by \texttt{Im-Comm-opt} in the figures, computes the communication power using the necessary optimality conditions given in \eqref{eq:theta-Z2}-\eqref{eq:grad-a}. In this case, the covariance of the communication signal at time $t$ takes the form $S_t = U\Lambda_t U^\top$, where $\Lambda_t = a_t H^{-1}$ and $a_t$ determined from the necessary optimality conditions. The second implementation determines the communication power heuristically (\texttt{Im-Comm-heu}). In particular, we set $\Lambda_t = \theta^t I$ at time $t$, where $\theta=0.88$ in our experiments.

\begin{figure*}[t]
	\centering
	\begin{subfigure}[b]{0.32\textwidth}
		\centering
		\includegraphics[width=\textwidth]{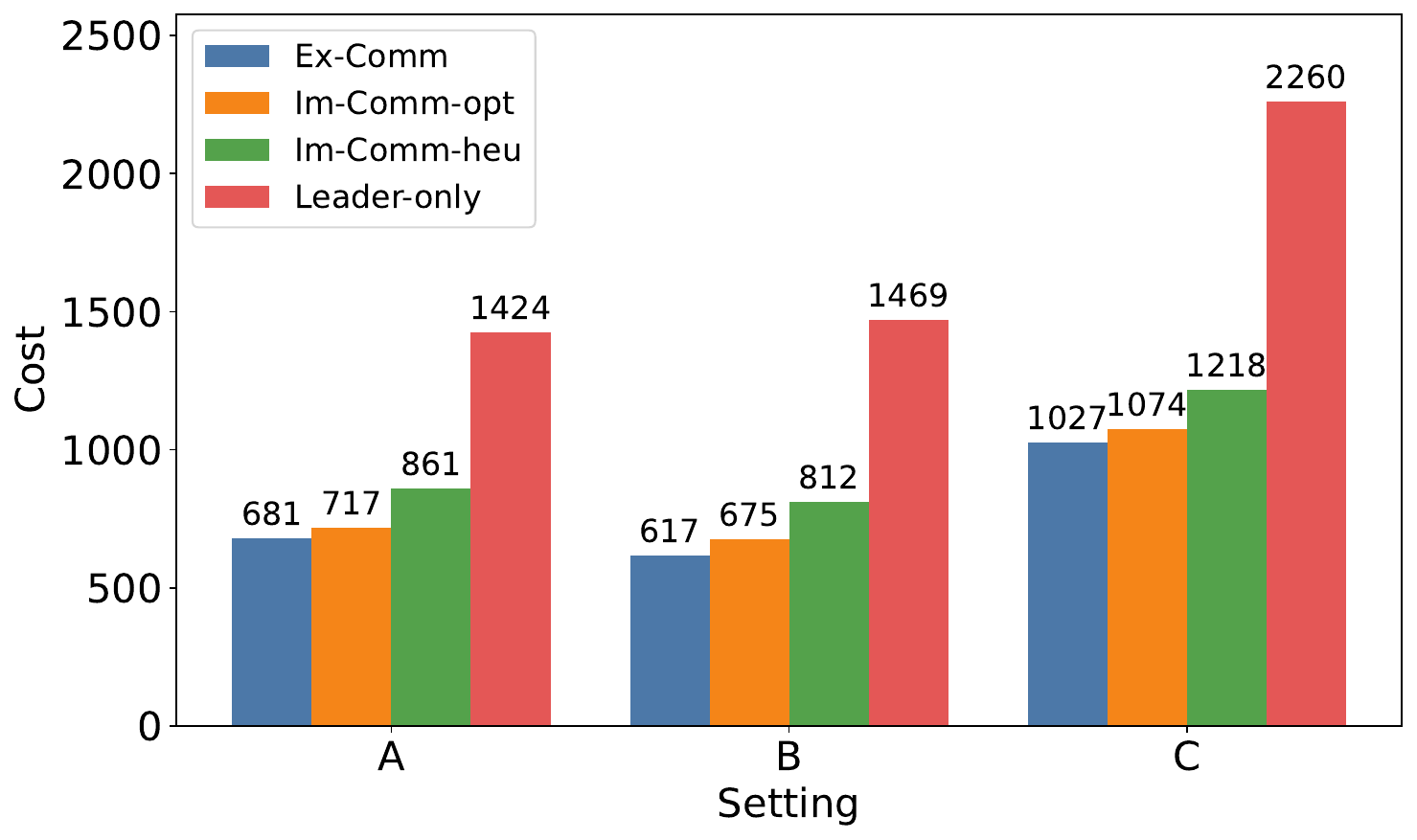}
		\caption{Control cost comparison} 
		\label{fig:cost_fa}
	\end{subfigure}
	\begin{subfigure}[b]{0.32\textwidth}  
		\centering 
		\includegraphics[width=\textwidth]{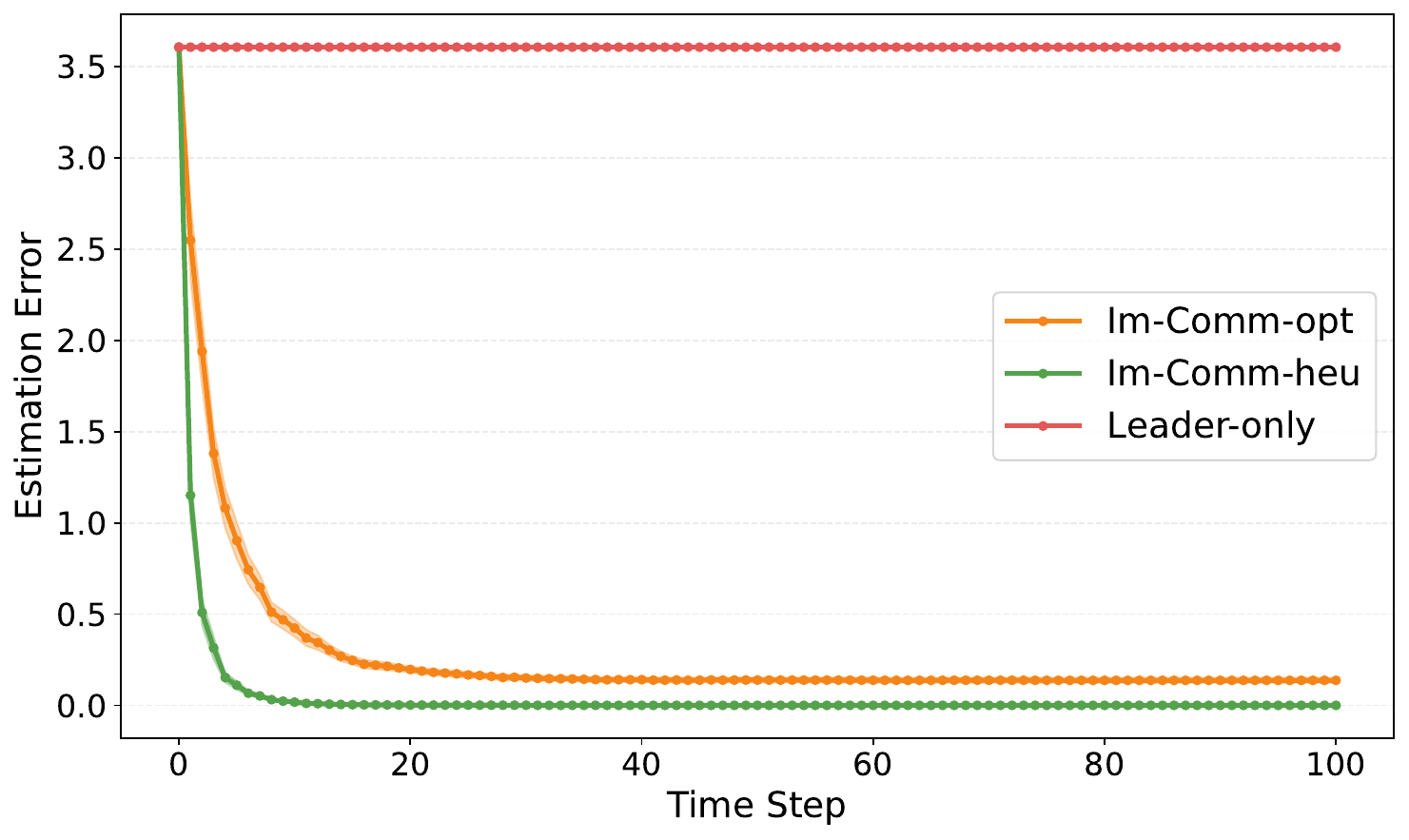}
		\caption{Estimation error over time}  
		\label{fig:et_fa}
	\end{subfigure}
	\begin{subfigure}[b]{0.32\textwidth}  
		\centering 
		\includegraphics[width=\textwidth]{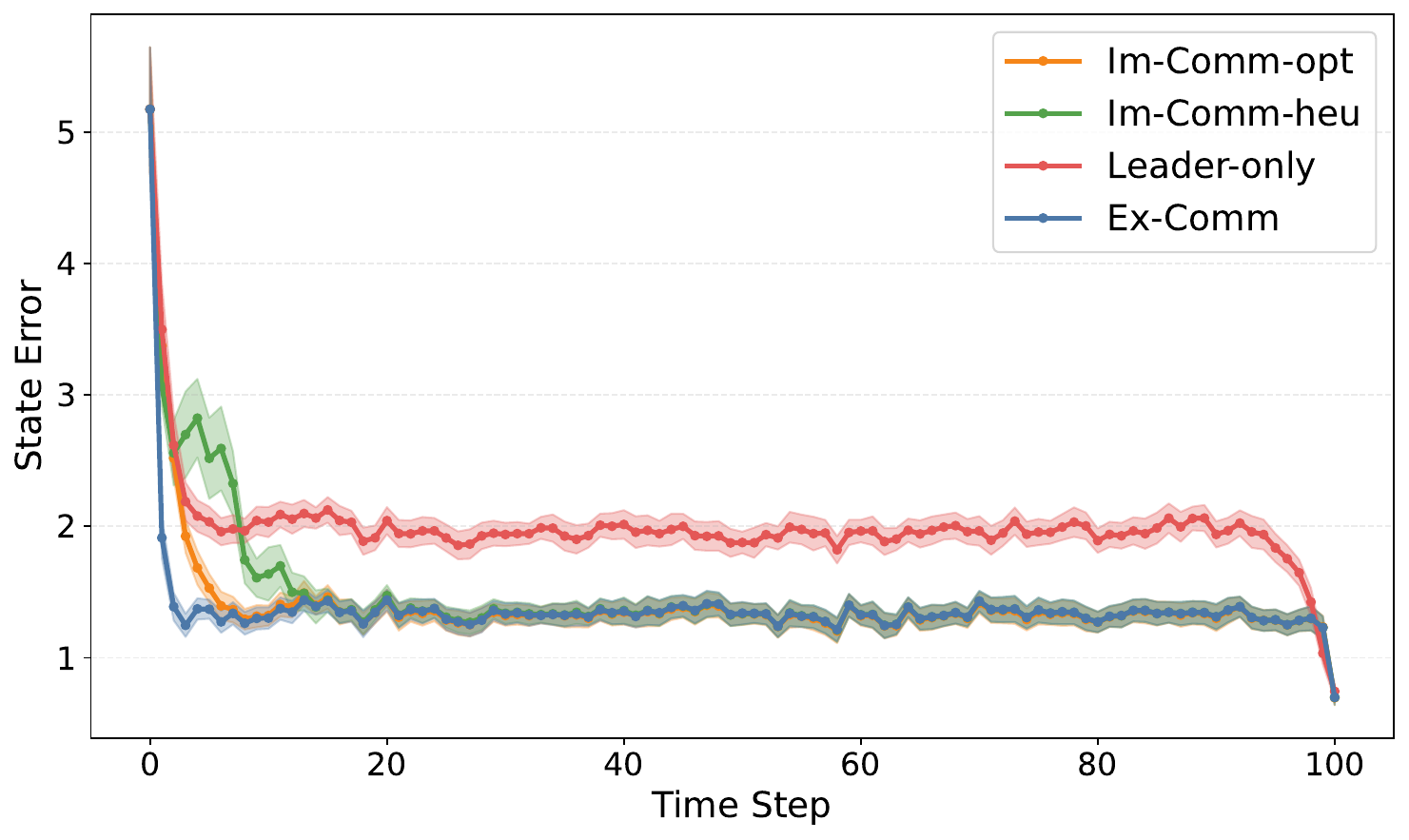}
		\caption{Distance to the target state over time}  
		\label{fig:zt_fa}
	\end{subfigure}
	\caption{Performance of the implicit coordination scheme in the fully actuated setting. \texttt{Ex-Comm}: control with explicit communication; \texttt{Im-Comm}: control with implicit communication, with suffixes \texttt{opt} and \texttt{heu} indicating optimally and heuristically chosen communication power, respectively; \texttt{Leader-only}: only the leader controls the system, while the follower applies zero input. In (a), different settings correspond to different target states.}
	\label{fig:fullact}
	
\end{figure*}

We first compare the overall control costs of all the above methods in Fig.~\ref{fig:cost_fa}, where different settings correspond to different target states. Specifically, target states in settings A, B, and C are randomly chosen as $[-1, 2,  2,  -2], [3, -1,  2,  1]$, and $[1, -2,  -3,  3]$, respectively. As shown in the figure, \texttt{Leader-only} results in a large control cost, typically more than twice that of \texttt{Ex-Comm}. In contrast, incorporating implicit communication to coordinate the two controllers significantly reduces the control cost. When the communication power is determined optimally, \texttt{Im-Comm-opt} achieves a control cost that is close to the cost of \texttt{Ex-Comm}. Even when the communication power is chosen heuristically, \texttt{Im-Comm-heu} still achieves good performance, demonstrating that the proposed implicit coordination scheme is robust to suboptimal choices of communication power. 

We next examine how the follower's estimation error evolves over time. Fig.~\ref{fig:et_fa} illustrates the trace of  the follower's estimation error covariance (i.e., $\tr(\Sigma_t)$) in setting A. The estimation error under \texttt{Leader-only} remains constant over time, since there does not exist any form of communication between the follower and the leader. In contrast, the estimation error of the two implicit communication methods decreases rapidly to a small value (determined by $\epsilon$ in the boundary condition \eqref{eq: bd-con}), indicating that the follower can accurately estimate the target state after only a few time steps. This observation verifies Theorem~\ref{thm:covet} and explains why implicit communication can reduce the overall control cost. Interestingly, the better-performing method, \texttt{Im-Comm-opt}, does not reduce the estimation error as quickly as \texttt{Im-Comm-heu}. This is because faster error reduction typically requires larger communication power, which in turn increases the control cost. Therefore, communication power and estimation accuracy must be carefully balanced.

We also compare the distance to the target state, defined as $||z_t|| = ||x_*-x_t||$, under the above methods.  Fig.~\ref{fig:zt_fa} shows the 2-norm of state error $z_t$ over time.
Since the leader is fully actuated, even \texttt{Leader-only} can eventually drive the system to the target state (the state error at the final time step is not exactly zero due to process noise). However, \texttt{Leader-only} yields relatively large state errors for most of the time horizon, which is the main reason for its high control cost. In contrast, the two implicit communication methods produce state error trajectories that are close to that of \texttt{Ex-Comm}. It is worth noting that, although \texttt{Im-Comm-heu} decreases the estimation error faster than \texttt{Im-Comm-opt}, its state error decreases more slowly. This is because, from the control perspective, the communication signal acts as additional noise. Since \texttt{Im-Comm-heu} uses larger communication power in the early stages, it introduces stronger disturbances to the control system, resulting in larger state errors.

\subsection{Under-actuator Setting}

\begin{figure*}[t]
	\centering
	\begin{subfigure}[b]{0.32\textwidth}
		\centering
		\includegraphics[width=\textwidth]{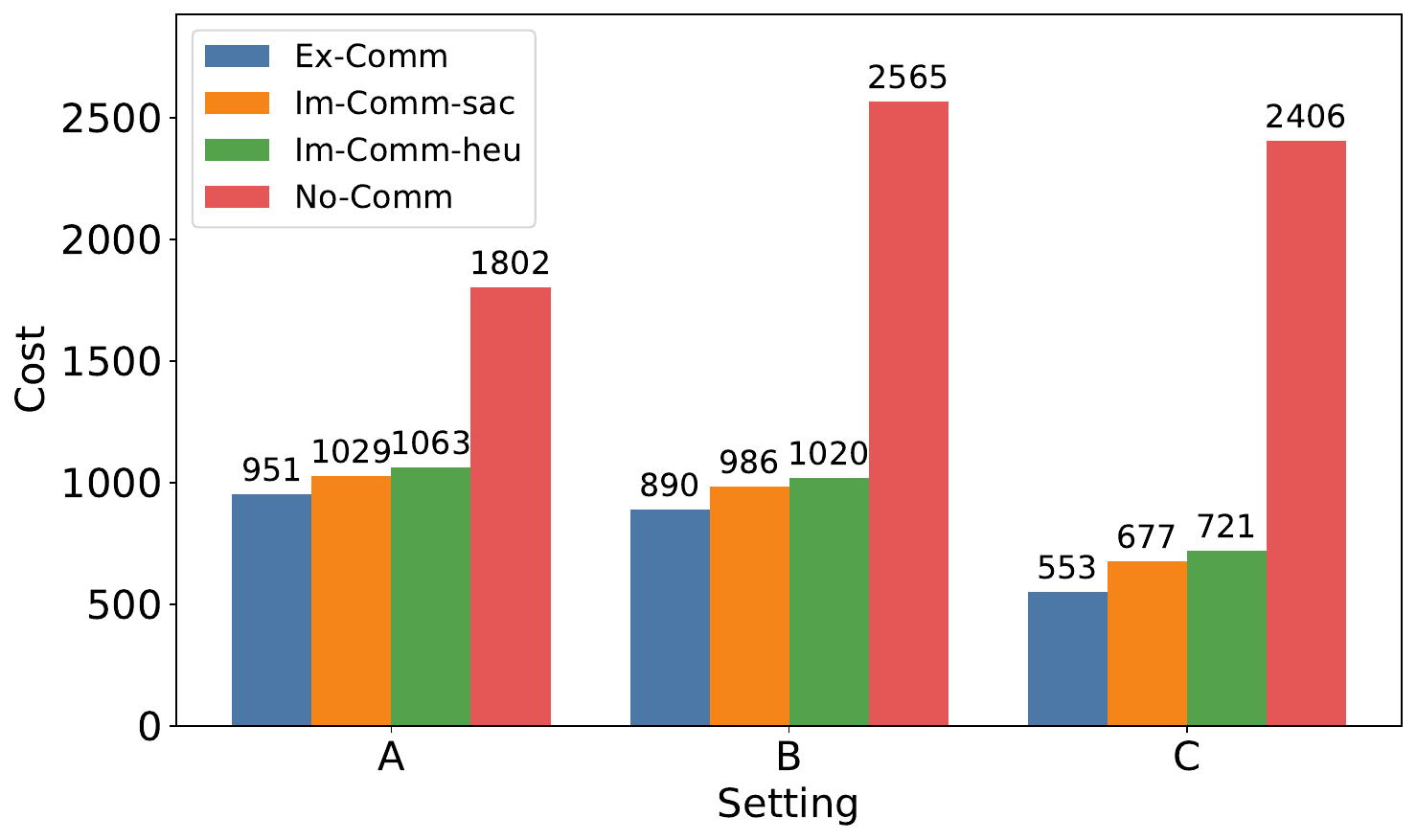}
		\caption{Control cost comparison} 
		\label{fig:cost_ua}
	\end{subfigure}
	\begin{subfigure}[b]{0.32\textwidth}  
		\centering 
		\includegraphics[width=\textwidth]{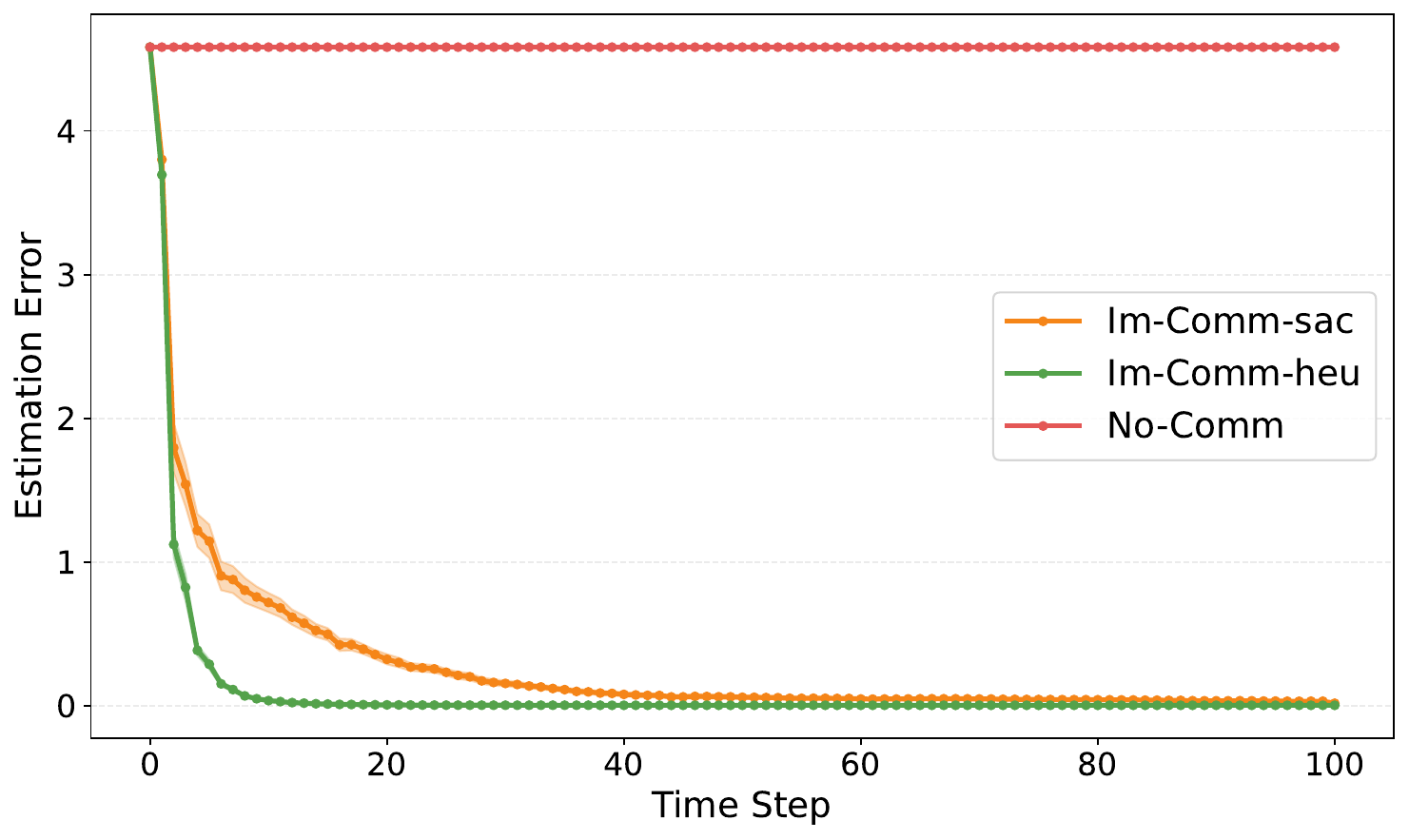}
		\caption{Estimation error over time}  
		\label{fig:et_ua}
	\end{subfigure}
		\begin{subfigure}[b]{0.32\textwidth}  
		\centering 
		\includegraphics[width=\textwidth]{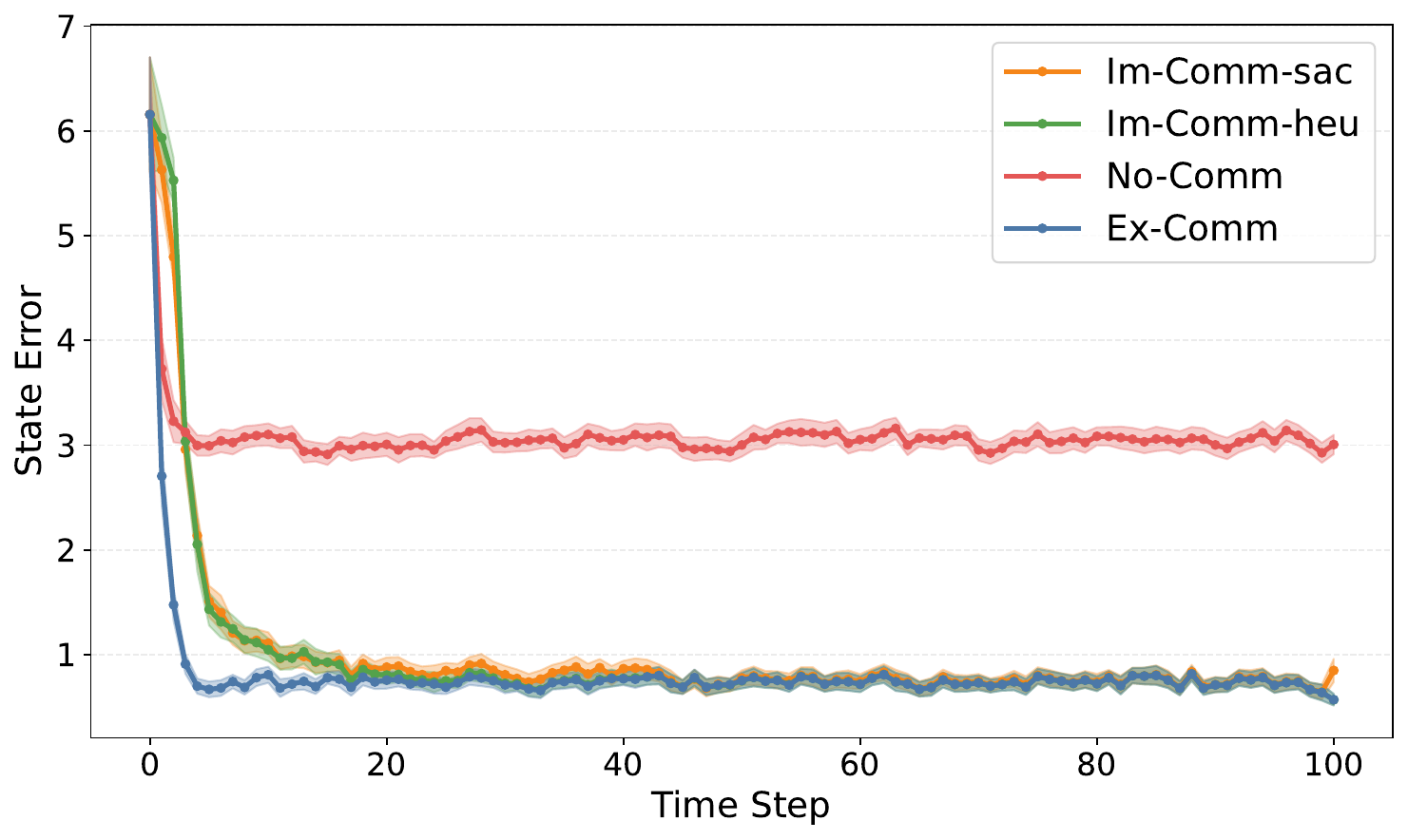}
		\caption{Distance to the target state over time}  
		\label{fig:zt_ua}
	\end{subfigure}
	\caption{Performance of the implicit coordination scheme in the under-actuated setting. \texttt{Im-Comm-sac}: control with implicit communication, where the communication power is selected using the SAC algorithm; \texttt{No-Comm}: control without communication.}	
	\label{fig:underact}
	
\end{figure*}

The proposed method is also evaluated on a under-actuated system defined as
\begin{align*}
	A = \begin{bmatrix}
		1.5& 0.2& 0& 0 \\
		0& 2.2& 0.5& 0.3 \\
		0& 0.2& 0.9& 0.4 \\
		0.2& 0& 0.3& 0.7
	\end{bmatrix},
	B_1 = \begin{bmatrix}
		1& 0 \\
		0& 0 \\
		0& 1.2 \\
		0& 0
	\end{bmatrix},
	B_2 = \begin{bmatrix}
		0& 0 \\
		1& 0 \\
		0& 0 \\
		0& 1.5
	\end{bmatrix}.
\end{align*}
The noise covariance is $W=\dig(0.1, 0.1, 0.1, 0.1)$. Moreover, the matrices that define the quadratic cost are specified as $F=\dig(2,1,1,2)$, $F_n = 10 I, G_1=\dig(2,2)$, and $G_2 = \dig(3,3)$. The covariance of the target state is $\Sigma_0 = \dig(5,5,5,5)$.

As in the experiments in the fully actuated setting, we choose \texttt{Ex-Comm} as the first benchmark.
The second benchmark is control without communication (\texttt{No-Comm}), where the follower has no communication---either explicitly or implicitly---with the leader. As a result, the follower never learns the correct target state and thus applies the control law that targets the origin. For the implicit communication method, we again consider two implementations. The first, \texttt{Im-Comm-heu}, selects the communication power heuristically, as in the fully actuated experiments; the second, \texttt{Im-Comm-sac}, determines the communication power using the soft actor-critic (SAC) algorithm. In particular, since the dynamics of the communication power optimization are complex, model-based methods such as iLQR require computing gradients, which is computationally demanding. We therefore adopt a model-free approach: we first construct an environment simulator using the dynamic equations \eqref{eq: Zt-under}-\eqref{eq: omegat-under}, and then apply the SAC algorithm to learn a policy by interacting with the simulator rather than explicitly computing gradients.

The control costs under different methods are compared in Fig.~\ref{fig:cost_ua}, where settings A, B, and C correspond to randomly chosen target states $[-1, 2,  2,  -2], [1, -1,  -2,  3]$, and $[2, -2,  3,  2]$, respectively. Fig.~\ref{fig:et_ua} and Fig.~\ref{fig:zt_ua} illustrate how the estimation error and state error evolves over time in setting C. As expected, \texttt{No-Comm} results in a large control cost because the two controllers can not drive the system to the target state without communication. Its state error remains at a high level throughout the time horizon, as shown in Fig.~\ref{fig:zt_ua}. In contrast, the proposed implicit communication methods result in small control costs that are close to the lower bound achieved by \texttt{Ex-Comm}. In this group of experiments, our implicit communication method again demonstrates strong robustness: the performance of \texttt{Im-Comm-heu} is vey close to that of \texttt{Im-Comm-sac}, in terms of both the control cost and the state error trajectory. This indicates that near-optimal control performance can be achieved even with a simple heuristic allocation of communication power.

\section{Conclusion and Future Work} \label{sec:conclusion}
In this paper, we proposed an ICoCo framework for decentralized LQG control, in which communication is coupled with control not only at the level of the message, but also at the level of the channel. Specifically, we studied a setting with asymmetric information, where a leader and a follower must coordinate to drive a linear system to a target state. The target state, however, is known only to the leader, and no explicit communication channel is available for conveying this information to the follower. Motivated by the idea of implicit communication, we proposed a coordination scheme in which the control plant itself serves as an implicit communication channel. Under this scheme, the leader encodes the target state into its control input by incorporating an additive Gaussian signaling term, while the follower progressively refines its estimate of the target state from the observed state evolution. We proved that, as long as the covariance matrix of the signaling term remains PD, the follower’s estimation error decreases monotonically to zero as time goes to infinity. Consequently, the follower is able to coordinate with the leader increasingly well over time and eventually steer the system to the target state. To minimize the overall control cost, we formulated the communication power design problem as an optimal control problem. In the fully actuated leader setting, we derived necessary conditions for the optimal solution. In the under-actuated leader setting, although such necessary conditions are unavailable, the problem can still be formulated and solved numerically. Finally, experimental results demonstrated the effectiveness of the proposed scheme, showing that it enables the two agents to achieve the target state with a control cost close to that attainable with explicit communication.

Beyond these technical results, this work reveals an interesting interplay between communication and control in multi-agent systems. Enabling implicit communication incurs an additional cost for the leader, as the leader must deviate from its nominally optimal policy to convey information. Yet this signaling improves coordination between the agents and can thereby reduce the overall team cost. In other words, a local sacrifice by the leader yields a global benefit for the team. 

More broadly, our work highlights the great potential of ICoCo designs for decentralized control. Many promising directions remain for future research, both theoretically and practically. On the theoretical side, we have shown that, for the problem considered here, implicit communication can be explicitly designed with a simple structure. A natural question is whether such a design remains effective under other control tasks and information patterns. For instance, one important extension is the case where the leader and the follower have different noisy observations of the system state. In this setting, coordination becomes substantially more challenging because the leader no longer has perfect feedback of the communication process, and it is unclear whether exchanging noisy observations is beneficial.  On the practical side, it is also of significant interest to investigate how the proposed ICoCo framework can be deployed in real-world applications.

\appendix
\subsection{Proof of Lemma 1}
\begin{IEEEproof}[Proof of Lemma~\ref{lem:estimation}]
	By definition,
	\begin{align*}
		e_{t+1} =  e_{t} - \hat{e}_{t}& = e_{t} -  \Sigma_{t}^{\frac{1}{2}} S_t^{\frac{1}{2}} Q^\top B_1^\top (B_1QS_t Q^\top B_1^\top + W)^{-1}  (B_1 QS_t^{\frac{1}{2}} \Sigma_{t}^{-\frac{1}{2}} e_{t} + w_t)  \\
		& = e_{t} -  \Sigma_{t}^{\frac{1}{2}} S_t^{\frac{1}{2}} Q^\top B_1^\top W^{-\frac{1}{2}} (W^{-\frac{1}{2}}B_1 QS_t Q^\top B_1^\top W^{-\frac{1}{2}} + I)^{-1} W^{-\frac{1}{2}} (B_1 QS_t^{\frac{1}{2}} \Sigma_{t}^{-\frac{1}{2}} e_{t} + w_t).
	\end{align*} 
	For notation simplicity, let ${B}_w \triangleq W^{-\frac{1}{2}}B_1$. Then
	\begin{align*}
		e_{t+1} & = e_{t} -  \Sigma_{t}^{\frac{1}{2}} S_t^{\frac{1}{2}} Q^\top B_w^\top (B_w QS_t Q^\top B_w^\top + I)^{-1} W^{-\frac{1}{2}} (B_1 QS_t^{\frac{1}{2}} \Sigma_{t}^{-\frac{1}{2}} e_{t} + w_t)  \\
		& \overset{(a)}{=}  \Sigma_{t}^{\frac{1}{2}} V_t  \Sigma_{t}^{-\frac{1}{2}} e_{t} - \Sigma_{t}^{\frac{1}{2}} S_t^{\frac{1}{2}}Q^\top B_w^\top (B_wQS_t Q^\top B_w^\top + I)^{-1} W^{-\frac{1}{2}} w_t, \\
		& \overset{(b)}{=}   \Sigma_{t}^{\frac{1}{2}} V_t  \Sigma_{t}^{-\frac{1}{2}} e_{t} - \Sigma_{t}^{\frac{1}{2}} V_t S_t^{\frac{1}{2}} Q^\top  B_w^\top  W^{-\frac{1}{2}} w_t \\
		& = \Sigma_{t}^{\frac{1}{2}} V_t  \Sigma_{t}^{-\frac{1}{2}} e_{t} - \Sigma_{t}^{\frac{1}{2}} V_t S_t^{\frac{1}{2}} Q^\top  B_1^\top  W^{-1} w_t
	\end{align*} 
	where (a) follows from the matrix inversion lemma and the assumption that $S_t$ and $(B_wQ)^\top B_wQ$ can both be diagonalized by the unitary matrix $U$:
	\begin{align*}
		V_t & \triangleq  I - S_t^{\frac{1}{2}}Q^\top B_w^\top (B_wQ S_t Q^\top B_w^\top + I)^{-1} B_w QS_t^{\frac{1}{2}} \\
		& = (I + S_t^{\frac{1}{2}}Q^\top B_w^\top B_wQS_t^{\frac{1}{2}})^{-1} \\
		& = (I + S_t^{\frac{1}{2}}Q^\top B_1^\top W^{-1} B_1QS_t^{\frac{1}{2}})^{-1} \\
		& = U(I + \Lambda_t H)^{-1} U^\top \\
		& \triangleq U \hat{V}_t U^\top,
	\end{align*}
	and (b) follows from the push-through identity:
	\begin{align*}
		S_t^{\frac{1}{2}}Q^\top B_w^\top (B_wQS_tQ^\top B_w^\top + I)^{-1} & = (I + S_t^{\frac{1}{2}} Q^\top B_w^\top B_wQ S_t^{\frac{1}{2}} )^{-1} S_t^{\frac{1}{2}} Q^\top B_w^\top \\
		& = U(I + \Lambda_t H)^{-1} U^\top S_t^{\frac{1}{2}} Q^\top B_w^\top \\
		& = V_t S_t^{\frac{1}{2}} Q^\top B_w^\top.
	\end{align*}
	Since $\Lambda_t$ and $H$ are diagonal matrices, it is easy to see that $\hat{V}_t$ is also a diagonal matrix with the $i$-the diagonal element given by 
	\begin{align*}
		\hat{V}_t(i) = \frac{1}{1 + \Lambda_t(i) H(i)}.
	\end{align*}
	Clearly, the distribution of $e_{t+1}$ is $\mathcal{N}(0, \Sigma_{t+1})$, where 
	\begin{align*}
		\Sigma_{t+1} & = \Sigma_{t}^{\frac{1}{2}}{V}_t{V}_t  \Sigma_{t}^{\frac{1}{2}}  + \Sigma_{t}^{\frac{1}{2}}{V}_t S_t^{\frac{1}{2}} Q^\top B_w^\top B_wQ S_t^{\frac{1}{2}} {V}_t  \Sigma_{t}^{\frac{1}{2}} \\
		&=\Sigma_{t}^{\frac{1}{2}} U \hat{V}_t \hat{V}_t U^\top \Sigma_{t}^{\frac{1}{2}} + \Sigma_{t}^{\frac{1}{2}} U \hat{V}_t \Lambda_t H \hat{V}_t U^\top \Sigma_{t}^{\frac{1}{2}}  \\
		& = \Sigma_{t}^{\frac{1}{2}} U \hat{V}_t U^\top \Sigma_{t}^{\frac{1}{2}}. 
	\end{align*}
	This completes the proof.
	
\end{IEEEproof}

\subsection{Derivation of Equation \eqref{eq:theta-sigma}}
The gradient of the linear terms w.r.t. $\Sigma_t$ is trivial, hence we focus on the non-linear terms. Let
\begin{align*}
	h_1(\Sigma_t)\triangleq& \tr(\Sigma_{t}^{\frac{1}{2}} U(I+\Lambda_tH)^{-1}U^\top \Sigma_{t}^{\frac{1}{2}} \theta_{\Sigma,t+1}), \\
	h_2(\Sigma_t)\triangleq& \tr(\Lambda^{\frac{1}{2}}_t U^\top \Sigma_{t}^{\frac{1}{2}} L_{t+1}^\top \theta_{Z,t+1}Q_1 U),  \\
	h_3(\Sigma_t)\triangleq& \tr(\Lambda^{\frac{1}{2}}_t U^\top \Sigma_{t}^{\frac{1}{2}} (D_t + K_tL_t)^\top G \tilde{I}Q U). 
\end{align*}
First, we derive the gradient of $h_1$ w.r.t. $\Sigma_t$.
Let $X=\Sigma_{t}^{\frac{1}{2}}$, then $\Sigma_t = X^2$ and define
\begin{align*}
	h_1(\Sigma_t) = \bar{h}_1(X)\triangleq \tr(XV_tX \theta_{\Sigma,t+1}).
\end{align*}
The differential of $\bar{h}_1(X)$ w.r.t. $X$ is given by
\begin{align*}
	d\bar{h}_1 = \tr((dX) V_t X \theta_{\Sigma,t+1}) + \tr(X V_t (dX )\theta_{\Sigma,t+1}) = \tr((dX)[\theta_{\Sigma,t+1} X V_t +V_t X \theta_{\Sigma,t+1}])
\end{align*}
Note that $\theta_{\Sigma,t+1} X V_t +V_t X \theta_{\Sigma,t+1}$ is symmetric and that $d\Sigma_t = (dX)X + X(dX) $. Then by definition, the gradient of $h_1$ w.r.t. $\Sigma_t$, denoted by $\Theta_1=\nabla_{\Sigma_t} h_1$, is a symmetric matrix satisfying
\begin{align*}
	dh_1 & = \tr(\Theta_1 (d\Sigma_t)) =  \tr((dX)(X\Theta_1 +\Theta_1 X) ) \\
	& = d\bar{h}_1 = \tr((dX)[\theta_{\Sigma,t+1} X V_t +V_t X \theta_{\Sigma,t+1}]).
\end{align*} 
By matching matrices, we have
\begin{align*}
	\Sigma_{t}^{\frac{1}{2}} \Theta_1 + \Theta_1 \Sigma_{t}^{\frac{1}{2}} = \theta_{\Sigma,t+1} \Sigma_{t}^{\frac{1}{2}} V_t +V_t \Sigma_{t}^{\frac{1}{2}} \theta_{\Sigma,t+1}.
\end{align*}
This is the well-known Sylvester-Lyapunov equation, which has a unique solution $\Theta_1$ provided that $\Sigma_{t}^{\frac{1}{2}}$ is PD \cite{horn1994matrix}.

Next, we proceed to derive the gradient of $h_2$. As before, let $X=\Sigma_{t}^{\frac{1}{2}}$ and define
\begin{align*}
	\bar{h}_2(X) \triangleq \tr(X L_{t+1}^\top \theta_{Z,t+1}Q_1 S^{\frac{1}{2}}_t) = g_2(\Sigma_t). 
\end{align*}
Then the differential of $\bar{h}_2(X)$ w.r.t. $X$ is given by
\begin{align*}
	d\bar{h}_2 &=  \tr((dX) L_{t+1}^\top \theta_{Z,t+1}Q_1 S^{\frac{1}{2}}_t) \\
	&= \frac{1}{2} \tr((dX)(L_{t+1}^\top \theta_{Z,t+1}Q_1 S^{\frac{1}{2}}_t +   S^{\frac{1}{2}}_t Q_1^\top\theta_{Z,t+1}L_{t+1} )).
\end{align*}
We need the second line because $X$ is symmetric, hence only the symmetric part of $L_{t+1}^\top \theta_{Z,t+1}Q_1 S^{\frac{1}{2}}_t$ matters.
Now,
\begin{align*}
	dh_2 = \tr(\nabla_{\Sigma_t} h_2 (dX)) = \tr((dX)(X\nabla_{\Sigma_t} h_2 + \nabla_{\Sigma_t} h_2 X) )  = d\bar{h}_2.
\end{align*}
Matching matrices yields
\begin{align*}
	\Sigma_{t}^{\frac{1}{2}}\nabla_{\Sigma_t} h_2 + \nabla_{\Sigma_t} h_2 \Sigma_{t}^{\frac{1}{2}} = \frac{1}{2}(L_{t+1}^\top \theta_{Z,t+1}Q_1 S^{\frac{1}{2}}_t +   S^{\frac{1}{2}}_t Q_1^\top\theta_{Z,t+1}L_{t+1} ).
\end{align*}
Again, this Sylvester-Lyapunov equation has a unique solution $\nabla_{\Sigma_t} h_2$. 

Finally, using the same argument yields the gradient of $h_3$ w.r.t. $\Sigma_t$, which is determined by the following equation:
\begin{align*}
	\Sigma_{t}^{\frac{1}{2}}\nabla_{\Sigma_t} h_3 + \nabla_{\Sigma_t} h_3 \Sigma_{t}^{\frac{1}{2}} = \frac{1}{2}((D_t + K_tL_t)^\top G \tilde{I}Q S^{\frac{1}{2}}_t +   S^{\frac{1}{2}}_t (G \tilde{I}Q)^\top (D_t + K_tL_t)).
\end{align*}
Equation \eqref{eq:theta-sigma} can be easily derived by combining the above arguments.

\subsection{Proof of Lemma~\ref{lem:virtual-est}}
\begin{IEEEproof}[Proof of Lemma \ref{lem:virtual-est}]
	By definition,
	\begin{align*}
		\varepsilon_t = e_t - \hat{e}_t & = e_t  - \Sigma_t^{\frac{1}{2}}  P^\top_k S_t^{\frac{1}{2}} \Psi_1(\Psi_1 S_t \Psi_1 + \bar{W}_1)^{-1} (\Psi_1 S_t^{\frac{1}{2}} P_k \Sigma_t^{-\frac{1}{2}} e_t + \tilde{w}_t) \\
		& =  \Sigma_{t}^{\frac{1}{2}} \bar{V}^{(k)}_t  \Sigma_{t}^{-\frac{1}{2}} e_{t} - \Sigma_t^{\frac{1}{2}}  P^\top_k S_t^{\frac{1}{2}} \Psi_1(\Psi_1 S_t \Psi_1 + \bar{W}_1)^{-1}   \tilde{w}_t,
	\end{align*}
	where $\bar{V}^{(k)}_t $ is given by
	\begin{align*}
		\bar{V}^{(k)}_t  &\triangleq I -  P^\top_k S_t^{\frac{1}{2}}  \Psi_1(\Psi_1  S_t  \Psi_1 + \bar{W}_1)^{-1} \Psi_1 S_t^{\frac{1}{2}} P_k \\
		&= I - P^\top_k S_t^{\frac{1}{2}}  \Psi_1 \bar{W}^{-\frac{1}{2}}_1 ( \bar{W}^{-\frac{1}{2}}_1 \Psi_1 S_t^{\frac{1}{2}} P_k P_k^\top S_t^{\frac{1}{2}}  \Psi_1 \bar{W}^{-\frac{1}{2}}_1 + I)^{-1} \bar{W}^{-\frac{1}{2}}_1  \Psi_1 S_t^{\frac{1}{2}} P_k \\
		& \overset{(a)}{=} (I + P^\top_kS_t^{\frac{1}{2}}  \Psi_1 \bar{W}^{-1}_1 \Psi_1  S_t^{\frac{1}{2}} P_k)^{-1} \\
		& \overset{(b)}{=}(I +  P^\top_k U_1 \Lambda_t \Pi_1 U_1^\top P_k )^{-1} \\
		& = (I +  U_\tau P^\top_k \Lambda_t \Pi_1  P_k U_\tau^\top )^{-1} \\
		& = U_\tau (I +  P^\top_k \Lambda_t \Pi_1  P_k  )^{-1} U_\tau^\top \\
		& \triangleq U_\tau \tilde{V}^{(k)}_{t} U^\top_\tau,
	\end{align*}
	where $U_k\triangleq \mathtt{diag}(U_1, \cdots, U_1)\in \mathbb{R}^{d_0\times d_0}$. Here, (a) follows from the matrix inversion lemma, and (b) follows from the fact that $S_{t,i}$ and $\Psi_1 \bar{W}^{-1}_1 \Psi_1 $ are simultaneously diagonalizable by $U_1$; 
	Note that $\tilde{V}^{(k)}_t=(I +  P^\top_k \Lambda_t \Pi_1  P_k  )^{-1}$ is a $d_0\times d_0$ block diagonal matrix of the form $\tilde{V}^{(k)}_t = \mathtt{diag}(I_{(k-1)r}, \tilde{V}_{t},I_{(\tau-k)r})$, where $I_j$ denotes the $j\times j$ identity matrix, and $\tilde{V}_{t}=\Lambda_{t}\Pi_1$ is a diagonal matrix with the $j$-th diagonal entry given by
	\begin{align*}
		\tilde{V}_{t}(j) = \frac{1}{1+ \Lambda_{t}(j) \Pi_1(j)}.
	\end{align*}
	Furthermore,
	\begin{align*}
		P^\top_kS_t^{\frac{1}{2}}  \Psi_1(\Psi_1 S_t \Psi_1 + \bar{W}_1)^{-1} &= P^\top_k S_t^{\frac{1}{2}}  \Psi_1 \bar{W}^{-\frac{1}{2}}_1 ( \bar{W}^{-\frac{1}{2}}_1 \Psi_1 S_t^{\frac{1}{2}} P_k P_k^\top S_t^{\frac{1}{2}}  \Psi_1 \bar{W}^{-\frac{1}{2}}_1 + I)^{-1} \bar{W}^{-\frac{1}{2}}_1 \\
		& \overset{(a)}{=} (I + P^\top_kS_t^{\frac{1}{2}}  \Psi_1 \bar{W}^{-1}_1 \Psi_1  S_t^{\frac{1}{2}} P_k)^{-1} P^\top_kS_t^{\frac{1}{2}} \Psi_1\bar{W}^{-1}_1 \\
		& = \bar{V}^{(k)}_t  P^\top_k S_t^{\frac{1}{2}}  \Psi_1 \bar{W}^{-1}_1,
	\end{align*}
	where (a) follows from the push-through identity. Putting the above derivations together yields
	\begin{align*}
		\varepsilon_t =\Sigma_{t}^{\frac{1}{2}} \bar{V}^{(k)}_t  \Sigma_{t}^{-\frac{1}{2}} e_{t} -  \Sigma_{t}^{\frac{1}{2}}  \bar{V}^{(k)}_t  P^\top_k S_t^{\frac{1}{2}} \Psi_1 \bar{W}^{-1}_1 \tilde{w}_t.
	\end{align*}
	It follows that
	\begin{align*}
		\cov(\varepsilon_t)& = \Sigma_{t}^{\frac{1}{2}} \bar{V}^{(k)}_t \bar{V}^{(k)}_t \Sigma_{t}^{\frac{1}{2}} + \Sigma_{t}^{\frac{1}{2}}  \bar{V}^{(k)}_t  P^\top_k S_t^{\frac{1}{2}}  \Psi_1 \bar{W}^{-1}_1  \Psi_1  S_t^{\frac{1}{2}} P_k \bar{V}^{(k)}_t  \Sigma_{t}^{\frac{1}{2}}  \\
		& = \Sigma_{t}^{\frac{1}{2}} \bar{V}^{(k)}_t \bar{V}^{(k)}_t \Sigma_{t}^{\frac{1}{2}} + \Sigma_{t}^{\frac{1}{2}}  \bar{V}^{(k)}_t  U_\tau P^\top_k \Lambda_t {\Pi}_1 P_k U^\top_\tau \bar{V}^{(k)}_t  \Sigma_{t}^{\frac{1}{2}}  \\
		& =\Sigma_{t}^{\frac{1}{2}} \bar{V}^{(k)}_{t} \Sigma_{t}^{\frac{1}{2}} .
	\end{align*}
	This completes the proof.
		
\end{IEEEproof}

\bibliographystyle{IEEEtran}
\bibliography{reference}

\end{document}